\documentclass{amsart}
\RequirePackage{fix-cm}
\usepackage[utf8]{inputenc}

\usepackage{comment, amsmath, amscd, amsxtra, amssymb, amsthm, changepage, pifont, graphicx, xcolor, tikz, soul, longtable, hyperref, listings, algorithm, enumitem, appendix, graphics, epsfig, epic, latexsym, booktabs, mathtools, colortbl, afterpage, stmaryrd, nomencl}
\usepackage[color]{changebar}
\usetikzlibrary{cd}
\usepackage[margin=1.85cm]{caption}
\usetikzlibrary{calc, math}
\usepackage[margin=1.5in]{geometry}
\usepackage{todonotes}
\usepackage{appendix}
\allowdisplaybreaks

\numberwithin{figure}{section}

\newtheorem{theorem}{Theorem}[section]
\newtheorem{lemma}[theorem]{Lemma}

\theoremstyle{definition}
\newtheorem{definition}[theorem]{Definition}

\theoremstyle{remark}
\newtheorem{remark}[theorem]{Remark}

\DeclareMathOperator{\charac}{char}
\DeclareMathOperator{\disc}{disc}

\DeclareMathOperator{\GL}{GL}

\DeclareMathOperator{\ff}{Frac}
\DeclareMathOperator{\tr}{Tr}

\DeclareMathOperator{\res}{res}

\DeclareMathOperator{\red}{red}

\newcommand{\ZZ}{\mathbb{Z}}
\newcommand{\QQ}{\mathbb{Q}}

\newcommand{\FF}{\mathbb{F}}

\renewcommand{\AA}{\mathbb{A}}
\newcommand{\PP}{\mathbb{P}}

\title{Lifting low-gonal curves for use in Tuitman's algorithm}
\author{Wouter Castryck and Floris Vermeulen}
\begin{document}

\maketitle

\begin{abstract} 
Consider a smooth projective curve $\overline{C}$ over a finite field $\FF_q$, equipped with a simply branched morphism $\overline{C} \to \PP^1$ of degree $d \leq 5$. Assume $\charac \FF_q > 2$ if $d \leq 4$, and $\charac \FF_q > 3$ if $d=5$. In this paper we describe how to efficiently compute a lift of $\overline{C}$ to characteristic zero, such that it can be fed as input to Tuitman's algorithm for computing the Hasse--Weil zeta function of $\overline{C} / \FF_q$. Our method relies on the parametrizations of low rank rings due to Delone--Faddeev and Bhargava.
\end{abstract}

\section{Introduction}

\thispagestyle{empty}

About $20$ years ago, Kedlaya published an influential paper~\cite{kedlaya}, showing how one can employ Monsky--Washnitzer cohomology to efficiently compute Hasse--Weil zeta functions of hyperelliptic curves over finite fields having small odd characteristic. %\footnote{While his attention was limited to odd characteristic finite fields and curves carrying a rational Weierstrass point, these gaps were taken care of in~\cite{denefvercauteren,harrison}.} 
Its many follow-up works include several generalizations to geometrically larger classes of curves, first to superelliptic curves~\cite{gaudrygurel}, then to $C_{ab}$ curves~\cite{denefvercauteren} and then further to non-degenerate curves~\cite{castryckdenefvercauteren}, i.e., smooth curves in toric surfaces. A more significant step was taken in 2016, when Tuit\-man~\cite{tuitman1,tuitman2} published a Kedlaya-style algorithm that potentially covers arbitrary curves, and at the same time beats the methods from~\cite{castryckdenefvercauteren,denefvercauteren} in terms of efficiency.
Unfortunately, the user of Tuitman's algorithm is expected to provide a lift of the input curve to characteristic zero that meets the technical requirements from~\cite[Ass.\,1]{tuitman2}. Beyond non-degenerate curves, this is a non-trivial task.
%This is easy in the non-degenerate case, as explained below, but it is far from trivial in general. 
As a result, the exact range of applicability of Tuitman's method remains unclear. 
%In fact, up till now, his ideas have proven more fruitful in the Chabauty--Coleman and Chabauty--Kim methods for finding all rational points on certain curves over $\QQ$, see e.g.~\cite{cursedcurve}.

 A partial approach to lifting curves having gonality at most four was sketched in~\cite{castrycktuitman}, with concrete details being limited to curves of genus five. 
 In the current paper we present a different method, which is faster, works for curves of gonality at most five, and is much easier to implement.
 Concretely, we assume that we are given an absolutely irreducible curve over a finite field $\FF_q$ of characteristic $p > 2$, defined by a polynomial of the form
 \begin{equation} \label{eq:inputpol} \overline{f}_d(x)y^d + \overline{f}_{d-1}(x)y^{d-1} + \ldots + \overline{f}_0(x) \in \FF_q[x,y]
\end{equation}
for some $d \leq 5$. Moreover, the morphism $\overline{\varphi}$ from its non-singular projective model $\overline{C}$ to the projective line, induced by $(x,y) \mapsto x$, is assumed to be simply branched of degree $d$; in other words, all fibers of $\overline{\varphi}$ should consist of either $d-1$ or $d$ geometric points. Finally, if $d=5$ then it is assumed that $p>3$. Then our method efficiently produces a lift satisfying the main requirement from~\cite[Ass.\,1]{tuitman2}, which therefore can be fed as input to Tuitman's algorithm, modulo Heuristic H discussed below. 
%More precisely, we present the following result:

In terms of moduli, the locus of genus $g$ curves admitting a simply branched morphism to $\PP^1$ of degree at most $5$ has dimension $\min\{2g + 5, 3g-3\}$ by a result of Segre~\cite{segre}. For $g = 6$ and $g \geq 8$ this exceeds the locus of non-degenerate curves (and hence the locus of curves for which point counting was previously feasible) by four dimensions, see~\cite{CV}. In particular, our lifting procedure applies to all sufficiently general curves of genus $g \leq 8$.

\begin{remark}
Expecting our curve to be given in the form~\eqref{eq:inputpol} is essentially equivalent to assuming \emph{knowledge} of an $\FF_q$-rational  degree $d$ morphism $\overline{C} \to \PP^1$ that is simply branched, in contrast with the assumptions from~\cite{castrycktuitman}.
%, but
%for most practical applications this seems not much of a restriction. 
%\todo{Referee M zegt hier iets over superelliptic curves...}. 
%(Ik denk dat hij eerder verwijst naar de strengheid van de voorwaarde van het simply branched zijn, maar daar gaat het hier niet over; ik heb simply branched tussen haakjes gezet)
If such a morphism to $\PP^1$ exists but is not known, then one can try %\footnote{The algorithm from~\cite{schichoschreyerweimann} was designed to work in characteristic $0$, and there are some issues when naively running it in finite characteristic, see e.g.~\cite[Rem.\,25]{castrycktuitman}; it is not clear to us how fundamental these are.} 
to resort to methods due to Schicho--Schreyer--Weimann~\cite{schichoschreyerweimann} or Derickx~\cite[\S2.3]{derickx} for finding one.
\end{remark}

\subsection*{Lifting strategy.} % Let $g$ be the genus of $\overline{C}$. 
Write $q = p^n$ and fix a degree $n$ number field $K$ in which $p$ is inert. Let $\mathcal{O}_K$ denote its ring of integers and identify $\FF_q$ with $\mathcal{O}_K / (p)$. To \emph{lift} the curve $\overline{C}$ means to produce a non-singular projective curve $C / K$ whose reduction mod $p$ is isomorphic to $\overline{C} / \FF_q$; necessarily, the genus of $C$ should be equal to that of $\overline{C}$. Our actual goal is to lift the morphism $\overline{\varphi}$, which means that we want to equip $C$ with a morphism $\varphi : C \to \PP^1$ reducing to $\overline{\varphi} : \overline{C} \to \PP^1$ mod $p$, up to isomorphism.
%Our goal is to \emph{lift} the morphism $\overline{\varphi} : \overline{C} \to \mathbb{P}^1$, which means to produce a non-singular projective curve $C / K$ having good reduction modulo $p$, together with a $K$-rational morphism $\varphi : C \to \PP^1$ , whose reduction modulo $p$ is birational to $\overline{C}$. 
Our approach to solving this problem is based on the parametrization of low rank rings by Delone and Faddeev~\cite[Prop.\,4.2]{gangrosssavin}, and Bhargava~\cite{bhargava4, bhargava5}, in combination with algorithms due to Hess for computing reduced bases~\cite{hess}. In doing so, we will find concrete, typically non-planar equations for $\overline{C}$ over $\FF_q$ that have ``free coefficients'', which can be lifted to $\mathcal{O}_K$ naively,\footnote{Lifting $\overline{a} \in \FF_q \setminus \{0\}$ \emph{naively} to $\mathcal{O}_K$ means: producing whatever element $a \in \mathcal{O}_K$ such that $a \bmod p = \overline{a}$.} in order to obtain a non-singular projective curve $C / K$ along with a morphism $\varphi : C \to \PP^1$ of the said kind. We refer to Section~\ref{sec:preliminaries} for a more elaborate discussion.

\begin{remark}
  In general, the polynomial~\eqref{eq:inputpol}, which defines a plane curve that is birationally equivalent with $\overline{C}$, is not liftable directly: there may be many singularities, which typically disappear when lifting the coefficients of~\eqref{eq:inputpol} naively to $\mathcal{O}_K$, causing an increase of the genus.
\end{remark}

\begin{remark}
  In Kedlaya's original algorithm, corresponding to the case $d = 2$, an implicit first step is to rewrite~\eqref{eq:inputpol} into Weierstrass form. Indeed, Weierstrass models have ``free coefficients" that can be lifted naively to $\mathcal{O}_K$, always resulting in a hyperelliptic curve over $K$ having the same genus. From now on we assume $d \geq 3$.
\end{remark}

%Our approach to solving the lifting problem is to mimic Schreyer's proof~\cite[Cor.\,6.8]{schreyer} of the unirationality of $\mathcal{H}_{g,d}$, the Hurwitz space of simply branched degree $d \leq 5$ covers of $\PP^1$ by curves of genus $g$. 
%In doing so, we will find concrete, typically non-planar equations for $\overline{C}$ over $\FF_q$ that have ``free coefficients'', which can be lifted to $\mathcal{O}_K$ naively,\footnote{Lifting $\overline{a} \in \FF_q \setminus \{0\}$ \emph{naively} to $\mathcal{O}_K$ means: producing whatever element $a \in \mathcal{O}_K$ such that $a \bmod p = \overline{a}$.} in order to obtain a non-singular projective curve $C / K$. This procedure also lifts the morphism $\overline{\varphi}$ to a morphism $\varphi : C \to \PP^1$. 
%whose set of branch points reduces modulo $p$ onto the set of branch points of $\overline{\varphi}$.
%As we will see, Schreyer's proof can be made effective using the  Delone--Faddeev~\cite[Prop.\,4.2]{gangrosssavin} and Bhargava~\cite{bhargava4,bhargava5} correspondences for cubic resp.\ quartic and quintic rings, in combination with algorithms due to Hess for computing reduced bases~\cite{hess}.

Through elimination of variables (i.e., projection) we then obtain a planar model of the form
$f_d(x)y^d + f_{d-1}(x)y^{d-1} + \ldots + f_0(x) = 0$,
for polynomials $f_i \in \mathcal{O}_K[x]$ which, in general, do not reduce to $\overline{f}_i$ mod $p$; here, the lifted morphism $\varphi$ again corresponds to $(x,y) \mapsto x$. The change of variables $y \leftarrow y / f_d(x)$ yields a monic defining equation 
 \begin{equation} \label{eq:liftedpol} Q(x,y) = y^d + f_{d-1}(x)y^{d-1} + \ldots + f_0(x) f_d(x)^{d-1},
\end{equation}
having the right shape to serve as input for Tuitman's algorithm. 
 All subsequent arithmetic in Tuitman's algorithm is done in the $p$-adic completion $\ZZ_q$ of $\mathcal{O}_K$ (or rather its fraction field $\QQ_q$), up to some finite $p$-adic precision. But for the lifting step it suffices to work over $\mathcal{O}_K$, and this has some implementation-technical advantages~\cite[Rmk.\,2]{castrycktuitman}. 

\subsection*{On Tuitman's assumption.} Let us discuss the specific requirements from \cite[Ass.\,1]{tuitman2} in more detail.
A first assumption concerns the polynomial $r(x) = \Delta/ \gcd(\Delta, d\Delta / dx)$ with $\Delta$ the discriminant of~\eqref{eq:liftedpol}, when viewed as a polynomial in $y$ over $\mathcal{O}_K[x]$:
\begin{itemize}
    \item[(a)] the discriminant of $r(x)$ is a unit in $\ZZ_q$.
\end{itemize}
Next, consider the ring $\mathcal{R} = \ZZ_q[x, 1/r, y]/(Q)$ and write $\QQ_q(x,y)$ for the field of fractions of $\mathcal{R} \otimes \QQ_q$ and $\FF_q(x,y)$ for the field of fractions of $\mathcal{R} \otimes \FF_q$.
A second assumption is that we know explicit matrices
\[ W_0 \in \GL_d(\ZZ_q[x, 1/r]) \qquad \text{and} \qquad W_\infty \in \GL_d(\ZZ_q[x^{\pm 1}, 1/r]) \]
such that, if we write $b_{j,0} = \sum_{i = 0}^{d-1} (W_0)_{i+1, j+1} y^i$ and $b_{j, \infty} = \sum_{i=0}^{d-1} (W_\infty)_{i+1, j+1} y^i$, then: %\todo{Referee M vraagt om integral basis te definieren}
\begin{itemize}
    \item[(b)] $\{ b_{0,0}, \ldots, b_{d-1,0} \}$ is an integral basis for $\QQ_q(x,y)$ over $\QQ_q[x]$ and its reduction mod $p$ is  an integral basis for $\FF_q(x,y)$ over $\FF_q[x]$,
    \item[(c)] $\{ b_{0,\infty}, \ldots, b_{d-1, \infty} \}$ is an integral basis for $\QQ_q(x,y)$ over $\QQ_q[x^{-1}]$ and its reduction mod $p$ is  an integral basis for $\FF_q(x,y)$ over $\FF_q[x^{-1}]$.
\end{itemize}
Finally, writing 
\[ \mathcal{R}_0 = \ZZ_q[x] b_{0,0} + \ldots + \ZZ_q[x] b_{d-1,0} \quad \text{and} \quad \mathcal{R}_\infty = \ZZ_q[x^{-1}] b_{0,\infty} + \ldots + \ZZ_q[x^{-1}] b_{d-1,\infty}, \] it is assumed that
\begin{enumerate}
    \item[(d)] the discriminants of the finite $\ZZ_q$-algebras $\left(\mathcal{R}_0/(r) \right)_{\red}$ and $\left(\mathcal{R}_\infty/(1/x)\right)_{\red}$ are units.
\end{enumerate}
 Here the subscript `red' means that we consider the reduced ring obtained by quotienting out the nilradical.\footnote{This takes into account the erratum pointed out in~\url{https://jtuitman.github.io/erratum.pdf}.}
 
 The geometric meaning of assumptions (a) and (d) is discussed in~\cite[Prop.\,2.3]{tuitman2}; see also~\cite[Rmk.\,2.3]{tuitman1}. They express that all branch points of $\varphi : C \to \PP^1$, as well as all points lying over these branch points, should be distinct mod $p$. In our context, these properties are automatic. Indeed, since $p > 2$ and
$\overline{\varphi}: \overline{C}\to \PP^1$ is simply branched, there is no wild ramification, hence the ramification divisor  
of $\varphi$ reduces mod $p$ to that of $\overline{\varphi}$. Thus, again because $\overline{\varphi}$ is simply branched, we see that the ramification points of $\varphi$ must reduce to $2g + 2d - 2$ distinct points  that take distinct images under $\overline{\varphi}$, as wanted; here $g$ denotes the genus of $\overline{C}$. We also see that $\varphi$ is simply branched as well.

 Assumptions (b) and (c), on the other hand, ask for an explicit description of our lift $\varphi : C \to \PP^1$ in terms of two affine patches $\varphi^{-1}(\PP^1 \setminus \{ \infty\})$ and $\varphi^{-1}(\PP^1 \setminus \{0\})$, glued together using $W = W_0^{-1}W_\infty$, that is compatible with reduction mod $p$. In Tuitman's own \texttt{pcc\_{}p} and \texttt{pcc\_{}q} code,\footnote{\texttt{https://github.com/jtuitman/pcc}, see \texttt{mat\_{}W0()} and \texttt{mat\_{}Winf()} in \texttt{coho\_{}p.m} and \texttt{coho\_{}q.m}.} the matrices $W_0$ and $W_\infty$ are found by computing integral bases for the function field extension $K(x) \subseteq K(C)$ defined by~\eqref{eq:liftedpol}, using the Magma intrinsic \texttt{MaximalOrderFinite()}, and hoping that these have good reduction mod $p$. There is a non-zero probability that this approach fails, in which case Tuitman's code outputs ``bad model for curve", but in practice this probability become negligible very rapidly as $q$ grows; see the tables in~\cite{castrycktuitman}. We therefore content ourselves with relying on the same bet, which we call Heuristic H:

\begin{definition}[informal]
The output~\eqref{eq:liftedpol} satisfies \emph{Heuristic H} if the associated integral bases of $K(C)$ over $K[x]$ and $K[x^{- 1}]$, computed using Magma as in Tuitman's implementation, meet the requirements from~\cite[Ass.\,1]{tuitman2}.
\end{definition}

Of course, if through some other method one 
manages to find integral bases with good reduction, then this would by-pass Heuristic H. In particular, if $d = 3$ then, as explained in Remark~\ref{goodbasis_d3}, such integral bases can be extracted as by-products of our lifting procedure. %It seems likely that the same remark applies to $d=4,5$, but extracting integral bases seems more technical in these cases.

\subsection*{Combined runtime.} 
The running time of our lifting procedure is strongly dominated by that of Tuitman's algorithm, as should be clear from the discussions in Sections~\ref{sec:trigonal},~\ref{sec:tetragonal} and~\ref{sec:pentagonal} below. We will therefore omit a detailed analysis, although it is crucial to note that lifting does not inflate the input size too badly. Concretely, if we let $\delta = \max_{0 \leq i \leq d} \deg \overline{f}_i$, then
\begin{itemize}
    \item the reader can check that all $f_i$'s are of degree $O(g)$, which in turn is $O(\delta)$ thanks to Baker's bound~\cite[Thm.\,2.4]{beelen},
    \item when lifting coefficients from $\FF_q$ to $\mathcal{O}_K$ naively, we can choose them to be of bit size $O(n \log q)$, and as a result the same asymptotic estimate applies to the size of the coefficients of the $f_i$'s,
    \item as discussed in~\cite[p.\,313-314]{tuitman2}, the matrices $W_0,W_\infty$ produced by the Magma intrinsic, as well as their inverses, involve $K(x)$-coefficients whose pole orders are in $O(\delta)$, as required by~\cite[Ass.\,2]{tuitman2}; for $d = 3$, the reader can check that the same bound applies to the integral bases from Remark~\ref{goodbasis_d3}. 
\end{itemize}
From~\cite[Thm.\,4.10]{tuitman2} it follows that  $\widetilde{O}(p \delta^4 n^3)$
bit operations suffice for computing the Hasse--Weil zeta function of any curve $\overline{C} / \FF_q$ of the form~\eqref{eq:inputpol}, where we recall our dependence on Heuristic H if $d = 4, 5$.

\subsection*{Practical performance.} This paper comes with an implementation of our lifting procedure in Magma~\cite{magma}, which can be found at~\url{https://homes.esat.kuleuven.be/~wcastryc/}. Appendix~\ref{app:report} reports on how the code performs in combination with Tuit\-man's implementation for computing Hasse--Weil zeta functions. As discussed there, this gives satisfactory results for $d = 3$ and $d=4$, leading to a substantial enlargement of the class of curves admitting fast computation of their zeta function (over finite fields with small odd characteristic). In degree $d=5$ the combined code is considerably slower. This is almost entirely due to the seemingly harmless ``elimination of variables" step, which is needed to put the lifted curve $C / K$ in the form~\eqref{eq:liftedpol} and which produces large hidden constants in the above $O(g)$ and $O(n \log q)$ estimates. %Currently, for $d=5$, one should mainly view Theorem~\ref{thm:maintheorem} as a theoretical contribution, although 
Nevertheless, here too, it is practically feasible to compute zeta functions in a non-trivial range.

%\subsection*{Magma code.} We have completely implemented our algorithms in Magma to solve the lifting problem for low-gonal curves. Concrete magma code, together with several examples can be found in the file accompanying this paper. 

%Our Magma code computing the liftable matrix $\overline{M}$ is rather page-consuming, hence we immediately refer to the file \texttt{lifting\_{}lowgonal\_{}5.m} accompanying this paper, which is discussed in more detail in the next section. It is worth pointing out that, when naively running the code in characteristic $3$, the output is trivial, i.e.\ it returns the zero matrix.

\subsection*{Tracks for future work.}
Besides mitigating the effect of variable elimination and getting rid of Heuristic H, a challenging goal is to dispose of the conditions on $p$ and of the condition that $\overline{\varphi}$ is simply branched. This seems to require changes to Tuitman's algorithm that are similar to how Denef and Vercauteren managed to make Kedlaya's algorithm work in even characteristic~\cite{denefvercauterenchar2}.
Also, as explained in Section~\ref{sec:preliminaries}, our naive lifting strategy using ``free coefficients" is closely related to Schreyer's proof~\cite[Cor.\,6.8]{schreyer} of the unirationality of $\mathcal{H}_{g,d}$, the moduli space of simply branched degree $d$ covers of $\PP^1$ by curves of genus $g$, for $d \leq 5$. Such unirationality results are known to be false for $d\geq 7$, where there is no hope for our strategy to work. This leaves $d=6$ as an interesting open case, on which several partial (positive) results have been proved by Geiss~\cite{geiss}, see~\cite[Fig.\,1]{schreyertanturri} for an overview. It seems worth investigating how Geiss' results combine with our approach.

\subsection*{Acknowledgements.} We thank Jan Tuitman and Yongqiang Zhao for several inspiring conversations, and the anonymous reviewers for their many helpful comments. This work is supported by CyberSecurity Research Flanders with reference VR20192203 and by  KU Leuven with references C14/17/083 and C14/18/067.

\section{Preliminaries} \label{sec:preliminaries}

\subsection*{Reduced bases and Maroni invariants}
Let $k$ be any field, which in the next sections will be specialized to $k = \FF_q$ and/or $k = K$.
 Consider a non-singular projective curve $C/k$ of genus $g$, along with a $k$-rational degree $d$ morphism $\varphi : C \to \PP^1$. Consider the inclusion of function fields $k(x) \subseteq k(C)$ corresponding to $\varphi$. Let $k[C]_0$, resp.\ $k[C]_\infty$, denote the integral closure of $k[x]$, resp.\ $k[1/x]$, inside $k(C)$. 

\begin{theorem} \label{thm:hessbasis}
There exist unique negative integers $r_1 \geq r_2 \geq \ldots \geq r_{d-1}$ for which there is a basis 
$1, \alpha_1, \ldots, \alpha_{d-1}$
of $k[C]_0$ over $k[x]$ such that 
$1, x^{r_1}\alpha_1, \ldots, x^{r_{d-1}}\alpha_{d-1}$
is a basis of $k[C]_\infty$ over $k[1/x]$.
\end{theorem}

See~\cite{hess} for a proof; it is standard to call $e_i = - r_i - 2$ the \emph{Maroni invariants} 
%(or \emph{scrollar invariants}) 
of $C$ with respect to $\varphi$ (e.g., if $\varphi$ is a degree $2$ cover, then there is just one Maroni invariant, namely $g-1$). A corresponding basis $1, \alpha_1, \ldots, \alpha_{d-1}$ is called a \emph{reduced basis}. 
In our cases of interest, the integers $r_i$ and an accompanying reduced basis can be computed efficiently: if $k$ is a finite field or a number field, then the Magma command \texttt{ShortBasis()} takes care of this.
%Thus, in our degree $d=2$ example from the introduction, $1,y$ is a reduced $\FF_q[x]$-basis of $\FF_q[\overline{C}]_0$, and $g-1$ is the (single) Maroni invariant.

\begin{remark}
In more geometric language, the integers $r_i$ are characterized by the sheaf decomposition
$\varphi_\ast \mathcal{O}_C \cong \mathcal{O}_{\PP^1} \oplus \mathcal{O}_{\PP^1}(r_1) \oplus \mathcal{O}_{\PP^1}(r_2) \oplus \ldots \oplus \mathcal{O}_{\PP^1}(r_{d-1})$
which, according to a theorem due to Grothendieck, is indeed unique. As a consequence to the Riemann--Roch theorem, the Maroni invariants satisfy the following basic properties: \emph{(i)} $-1 \leq e_1 \leq e_2 \leq \ldots \leq e_{d-1}$, \emph{(ii)} $e_1 + e_2 + \ldots + e_{d-1} = g - d + 1$, and \emph{(iii)} $e_{d-1} \leq (2g-2)/d$.
\end{remark}

\subsection*{Models with ``free coefficients"} 
As mentioned in the introduction, every cover $\varphi : C \to \PP^1$ of degree $3 \leq d \leq 5$ admits a non-singular projective model with ``free coefficients" that can be lifted naively from $\FF_q$ to $\mathcal{O}_K$. This follows from Schreyer's proof~\cite[Cor.\,6.8]{schreyer} of the unirationality of $\mathcal{H}_{g,d}$ for $d \leq 5$. The natural ambient space for this model is a \emph{rational normal scroll}, which can be obtained by gluing together
\[ (\PP^1 \setminus \{ \infty \}) \times \PP^{d-2} \qquad \text{and} \qquad (\PP^1 \setminus \{ 0 \}) \times \PP^{d-2} \]
in a non-standard way; the gluing depends on the Maroni invariants $e_1, \ldots, e_{d-1}$ of $C$ with respect to $\varphi$. We refer to~\cite{eisenbudharris,schreyer} for more details on this construction, as well as on the claims below. For the sake of conciseness we only describe what the model looks like on the left copy $\AA^1 \times \PP^{d-2}$, which we equip with coordinates $x, Y_1, \ldots, Y_{d-1}$.

First assume that $d = 3$. Then $C$ admits a defining equation of the form
\begin{equation} \label{liftable3} 
\sum_{l_1 + l_2 = 3} f_{l_1,l_2}(x) Y_1^{l_1} Y_2^{l_2} = 0
\end{equation}
with $\deg f_{l_1, l_2} \leq l_1 e_1 + l_2 e_2 + 4 - g$, such that $\varphi$ corresponds to projection on the $x$-coordinate. 
Conversely, every irreducible polynomial of the form~\eqref{liftable3} defines a curve having genus at most $g$; this can also be seen using Baker's bound~\cite[Thm.\,2.4]{beelen}, because the dehomogenization with respect to $Y_2$ is supported on the polygon
\begin{figure}[htbp]
\begin{center}
 \begin{tikzpicture}[scale=0.45]
   \draw [thick] (0,0) -- (10,0) -- (4,3) -- (0,3) -- (0,0);
   \node at (-0.3,-0.4) {\small $(0,0)$};
   \node at (10.3,-0.4) {\small $(2e_2 - e_1 +2,0)$};
   \node at (5.3,3.4) {\small $(2e_1 - e_2 +2,3)$};
   \node at (-0.3,3.4) {\small $(0,3)$};
 %  \node at (2.8,1.4) {$\Delta_\text{trig}(e_1,e_2)$};
 \end{tikzpicture}
\end{center}
\caption{\small Polygon describing covers of degree $3$. \label{fig:trig} }
\end{figure}
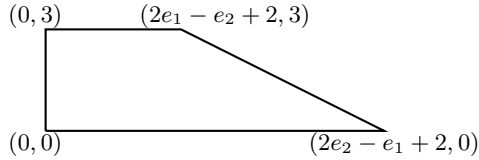
from Figure~\ref{fig:trig}. If equality holds then this polynomial defines a non-singular projective curve (on the entire rational normal scroll) and projection on the $x$-coordinate yields a degree $3$ morphism to $\PP^1$ whose associated Maroni invariants are $e_1, e_2$. 

Next, assume that $d = 4$. Then $C$ arises as the intersection of two surfaces defined by
\begin{equation} \label{liftable4} 
\sum_{l_1 + l_2 + l_3 = 2} f_{i, l_1,l_2, l_3}(x) Y_1^{l_1} Y_2^{l_2} Y_3^{l_3} = 0
\end{equation}
for $i = 1, 2$, where $\deg f_{i, l_1, l_2, l_3} \leq l_1e_1 + l_2e_2 + l_3e_3 - b_i$ for unique integers  $-1\leq b_1\leq b_2$ with $b_1+b_2 = g-5$,  called the Schreyer invariants of $C$ with respect to $\varphi$. 
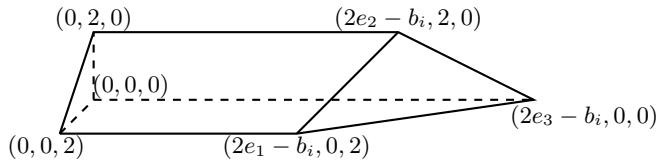
\begin{figure}[htbp]
\begin{center}
 \begin{tikzpicture}[scale=0.45]
   \draw [thick] (0,0) -- (7,0) -- (14,1) -- (10,3) -- (1,3) -- (0,0);
   \draw [thick] (7,0) -- (10,3);
   \draw [thick, dashed] (0,0) -- (1,1) -- (1,3);
   \draw [thick, dashed] (1,1) -- (14,1);
   \node at (-0.4,-0.4) {\small $(0,0,2)$};
   \node at (2.1,1.4) {\small $(0,0,0)$};   
   \node at (1,3.4) {\small $(0,2,0)$};
   \node at (10.3,3.4) {\small $(2e_2 - b_i,2,0)$};
   \node at (15.5,0.5) {\small $(2e_3 - b_i,0,0)$};
   \node at (7,-0.4) {\small $(2e_1 - b_i,0,2)$};
 %  \node at (6.1,2.2) {$\Delta_\text{tet}(e_1,e_2,e_3;b_i)$};
 \end{tikzpicture}
\end{center}
\caption{\small Polytope describing covers of degree $4$. \label{fig:tet} }
\end{figure}
Conversely, every irreducible such intersection defines a curve of genus at most $g$; this too can be seen using (a three-dimensional version of) Baker's bound~\cite[Thm.\,1]{khovanskii}, by noting that the dehomogenizations with respect to $Y_3$ are supported on the polytopes from Figure~\ref{fig:tet}. If equality holds then it concerns a non-singular projective curve, and projection on the $x$-coordinate defines a degree $4$ morphism to $\PP^1$ with associated Maroni invariants $e_1, e_2, e_3$ and Schreyer invariants $b_1, b_2$. 

Finally, assume $d=5$, which comes with five Schreyer invariants $b_1\leq \ldots \leq b_5$ summing up to $2g-12$. In this case $C$ can be viewed as the intersection of five hypersurfaces, which are all obtained from a single $5\times 5$ skew-symmetric matrix $M$ over $k[x][Y_1, Y_2, Y_3, Y_4]$ whose $(i,j)$-th entry is of the form 
\begin{equation} \label{liftable5}
M_{1,i,j}(x)Y_1 + M_{2,i,j}(x)Y_2 + M_{3,i,j}(x)Y_3 + M_{4,i,j}(x)Y_4
\end{equation}
with $M_{r,i,j}(x)\in k[x]$ of degree at most $e_r+b_i+b_j+6-g$. More precisely, our hypersurfaces are cut out by the five $4\times 4$ sub-Pfaffians\footnote{The square roots of the determinants of the five $4\times 4$ skew-symmetric submatrices.} of $M$. Conversely, whenever the $4 \times 4$ sub-Pfaffians of such a matrix define an irreducible curve, it has genus at most $g$. If equality holds then it concerns a non-singular projective curve, and projection on the $x$-coordinate defines a degree $5$ morphism to $\PP^1$ with Maroni invariants $e_1, e_2, e_3, e_4$ and Schreyer invariants $b_1, b_2, b_3, b_4, b_5$.

\subsection*{Lifting strategy revisited} In the next sections we show how results on ring parametrizations due to Delone--Faddeev~\cite[Prop.\,2.4]{gangrosssavin} and Bhargava~\cite{bhargava4,bhargava5} can be used to efficiently produce such a ``free coefficient" model for our input curve $\overline{C} / \FF_q$. Then, by the above discussion, and using that the genus cannot increase under reduction mod $p$, any naive coefficient-wise lift of this model to $\mathcal{O}_K$ will define a non-singular projective curve $C / K$ along with a morphism $\varphi : C \to \PP^1$ lifting $\overline{C}$ and $\overline{\varphi}$.

\begin{remark}
From a non-algorithmic viewpoint, the fact that the
Delone--Faddeev and Bhargava correspondences produce non-singular curves in rational normal scrolls might have been known to some specialists
(e.g., for $d=3$ this can be read in Zhao's Ph.D.\ thesis~\cite{zhao}).
\end{remark}

\section{Lifting curves in degree $d=3$} \label{sec:trigonal}

%In this and the next sections, we will describe how one can efficiently produce naively liftable models for low-degree covers of $\PP^1$. For this we will rely on the theory of ring parametrizations. 
For $R$ a PID, we recall that a \emph{ring of rank $d$} over $R$ is a commutative $R$-algebra which is free of rank $d$ as a module over $R$.
Every ring $S$ of rank $d$ over $R$ admits an $R$-basis of the form $1, \alpha_1, ..., \alpha_{d-1}$. This can be seen by
applying the structure theorem for finitely generated free modules over PIDs to the submodule $R \cdot 1$ of $S$.

\subsection*{Parametrizing cubic rings.} Let $R$ be a PID. Cubic rings over $R$ admit a parametrization using binary cubic forms over $R$, considered modulo a natural action by $\GL_2(R)$: for an element 
\[
A = \begin{pmatrix}
    a & b \\
    c & d
    \end{pmatrix} \in \GL_2(R),
\]
and $f = f_3Y_1^3+f_2Y_1^2Y_2+f_1Y_1Y_2^2+f_0Y_2^3$ a cubic form over $R$, we let
\[
A \ast f(Y_1,Y_2) = \frac{1}{\det A}f(a Y_1 + c Y_2, b Y_1 + d Y_2).
\]

\begin{theorem}[Delone--Faddeev]
There is a canonical bijection between the set of cubic $R$-rings up to isomorphism and binary cubic forms over $R$, modulo the action of $\GL_2(R)$.
\end{theorem}

For a proof, see e.g.~\cite[Prop.\,4.2]{gangrosssavin}. For use below we briefly describe how this bijection is constructed. Let $S$ be a cubic $R$-ring with basis $1, \alpha_1, \alpha_2$. By adding elements of $1\cdot R$ to $\alpha_1$ and $\alpha_2$ we can assume that $\alpha_1\alpha_2$ is in $R$. We call such bases \emph{normal}. Now write out the multiplication table of $S$:
\begin{equation} \label{eq: multiplication of cubic ring}
\begin{cases}
    \alpha_1\alpha_2 &= -g_0, \\
    \alpha_1^2 &= -g_1 + f_2\alpha_1 - f_3\alpha_2, \\
    \alpha_2^2 &= -g_2 + f_0\alpha_1 - f_1\alpha_2.
\end{cases}
\end{equation}
By associativity of $S$ we have $\alpha_1^2\cdot \alpha_2 = \alpha_1\cdot (\alpha_1 \alpha_2)$ and $\alpha_1\cdot \alpha_2^2=(\alpha_1 \alpha_2)\cdot\alpha_2$. This gives
\begin{equation}\label{eq: products in mutliplication table}
\begin{cases}
    g_0 &= f_0f_3, \\
    g_1 &= f_1f_3, \\
    g_2 &= f_0f_2,
\end{cases}
\end{equation}
so the $g_i$ are determined by the $f_i$. One then associates to $S$ the cubic form $f=f_3Y_1^3+f_2Y_1^2Y_2+f_1Y_1Y_2^2+f_0Y_2^3$. Conversely, given such a form $f$, associate to this the cubic ring, formally equipped with basis $1, \alpha_1, \alpha_2$ and multiplication defined by (\ref{eq: multiplication of cubic ring}) and (\ref{eq: products in mutliplication table}).
The $\GL_2(R)$-action on cubic forms corresponds precisely to changing one normal basis to another on the level of cubic rings.

\begin{remark} \label{remark_cubicfractionfield}
A cubic form $f = f_3Y_1^3+f_2Y_1^2Y_2+f_1Y_1Y_2^2+f_0Y_2^3$ is irreducible if and only if its associated cubic $R$-ring is a domain. In this case, we may describe it as the subring of
\begin{equation*} %\label{eq: function field delone faddeev}
\ff\left( \frac{R[y]}{(f_3y^3 + f_2y^2 + f_1y + f_0)}\right)
\end{equation*}
generated by $1, \alpha_1 = f_3y, \alpha_2 = - f_0y^{-1} = f_3y^2 + f_2y + f_1$. This point of view is especially nice when $R=k[x]$ for some field $k$. Indeed, then $f(y,1)=0$ defines a curve in $\AA^2$ over $k$ and the cubic ring associated to $f$ has as its field of fractions the function field of this curve.
\end{remark}
%For a binary cubic form $f(Y,Z)$ as above we may alternatively describe the cubic ring associated to it as the subring of 
%\begin{equation}\label{eq: function field delone faddeev}
%\ff\left( \frac{R[Y,Z]}{(f)}\right)
%\end{equation}
%generated by $1, -f_0Z, f_3Y$. This point of view is especially nice when $R=k[x]$. Indeed, then $f=0$ defines a trigonal curve in $\AA^1\times \PP^1$ and the cubic ring associated to $f$ has as its field of fractions the function field of this curve.

\subsection*{Lifting degree $3$ covers.} 
%We now specialize to $R= \FF_q[x]$. 
%where $k$ is a finite field. 
Consider the function field 
\[ \FF_q(\overline{C}) = \ff \left( \frac{\FF_q[x,y]}{(\overline{f}_3y^3+\overline{f}_2y^2+\overline{f}_1y +\overline{f}_0)} \right) \]
defined by our input polynomial, and consider the integral closure $\FF_q[\overline{C}]_0$ of $\FF_q[x]$ inside it; this is a cubic $\FF_q[x]$-ring. Let $e_1, e_2$ be the Maroni invariants of $\overline{C}$ with respect to $\overline{\varphi}$ and let $1, \alpha_1, \alpha_2$ be a corresponding reduced basis. After adding to $\alpha_1$ and $\alpha_2$ elements of $\FF_q[x]$ we may assume that this basis is normal. In more detail, if $\alpha_1\alpha_2 = a\alpha_1+b\alpha_2+c$, for $a,b,c\in \FF_q[x]$, then we replace $\alpha_1$ by $\alpha_1 - b$ and $\alpha_2$ by $\alpha_2-a$. This operation will not change the fact that the basis is reduced. Applying the Delone--Faddeev correspondence to this basis produces a new cubic form
\begin{equation*} %\label{eq:DFtrans}
\overline{f}(Y_1,Y_2) = \overline{f}_3Y_1^3+\overline{f}_2Y_1^2Y_2+\overline{f}_1Y_1Y_2^2+\overline{f}_0Y_2^3
\end{equation*}
whose coefficients we, abusingly, again denote by $\overline{f}_i$. 

\begin{lemma} \label{delonefaddeev_is_rightcurve} Let $\overline{f}$ be obtained through the Delone--Faddeev correspondence as above. Then this is a model for $\overline{C}$ of the form~\eqref{liftable3}. %In particular, any naive lift of $\overline{f}$ to $\mathcal{O}_K$ defines a smooth planar affine curve $C$ of genus $g$.

%The polynomial $\overline{f}(y,1) \in \FF_q[x,y]$ homogenizes to a polynomial of the form~\eqref{eq:homdef} with $m = 3$ and $b = 4 - g$, thereby defining a non-singular model of $\overline{C}$ inside $\PP(\mathcal{O}_{\PP^1}(e_1) \oplus \mathcal{O}_{\PP^1}(e_2))$.
\end{lemma}

\begin{proof}
Note that the curve $\overline{f}(y,1) = 0$ is indeed birationally equivalent with $\overline{C}$, 
in view of Remark~\ref{remark_cubicfractionfield}. Denote by $e_1, e_2$ the Maroni invariants of $\overline{C}$. Since $1,\alpha_1, \alpha_2$ is a reduced basis, the elements $1, x^{-e_1-2}\alpha_1, x^{-e_2-2}\alpha_2$ form a basis for $\FF_q[\overline{C}]_\infty$, the integral closure of $\FF_q[x^{-1}]$ inside $\FF_q(\overline{C})$.  Writing out the multiplication for this ring gives
\begin{equation*}
\begin{cases}
    x^{-e_1-e_2-4}\alpha_1\alpha_2 &= -x^{-e_1-e_2-4}\overline{f}_0\overline{f}_3, \\
    x^{-2e_1-4}\alpha_1^2 &= -x^{-2e_1-4}\overline{f}_1\overline{f}_3 + x^{-e_1-2}\overline{f}_2x^{-e_1-2}\alpha_1 - x^{-2e_1+e_2-2}\overline{f}_3x^{-e_2-2}\alpha_2, \\
    x^{-2e_2-4}\alpha_2^2 &= -x^{-2e_2-4}\overline{f}_0\overline{f}_2 + x^{-2e_2+e_1-2}\overline{f}_0x^{-e_1-2}\alpha_1 - x^{-e_2-2}\overline{f}_1x^{-e_2-2}\alpha_2.
\end{cases}
\end{equation*}
Since the coefficients of this table must be elements of $\FF_q[x^{-1}]$ we see that 
$\deg \overline{f}_i \leq (i-1)e_1+(2-i)e_2+2$
for $i=1,2$, hence $\overline{f}(y,1)$ is supported on the polygon from Figure~\ref{fig:trig}. %By construction, the lift $f$ is supported on the same Newton polygon as $\overline{f}$. Hence it follows from Baker's bound~\cite[Thm.\,2.4]{beelen} that $C$ also has genus $g$.
%In other words, $\overline{f}(y,1)$ is supported on the polygon from Figure~\ref{fig:trig}. But the lifted polynomial $f$ has the same Newton polygon, so the claim follows from Baker's bound~\cite[Thm.\,2.4]{beelen}.
\end{proof}

%Thus, we can proceed as outlined in Section~\ref{sec:existence}. Note that, in practice, there is no need to pass through the homogenization: it suffices to naively lift the polynomials $\overline{f}_i(x)$ to polynomials $f_i(x) \in \mathcal{O}_K[x]$ in a degree-preserving way, and output $f_3(x)y^3 + \ldots + f_0(x)$. This polynomial can then be fed as input to Tuitman's algorithm.

Thus we can proceed as follows. We compute a reduced basis for the function field $\FF_q(\overline{C})$ over $\FF_q[x]$, make it normal if needed, and apply the Delone--Faddeev correspondence to it to obtain a model $\overline{f}=0$ of the form~\eqref{liftable3}. As discussed in Section~\ref{sec:preliminaries}, any naive coefficient-wise lift of the polynomial $\overline{f}(y,1)$ to a polynomial $f = f_3y^3 + f_2y^2 + f_1y + f_0 \in \mathcal{O}_K[x]$ defines a good lift. %This is because the genus cannot increase under reduction modulo $p$. 
After making the polynomial $f$ monic as in~\eqref{eq:liftedpol}, it can be fed to Tuitman's algorithm to compute the zeta function of $\overline{C}$ over $\FF_q$. 

\begin{remark} \label{goodbasis_d3}
Our discussion also shows that
$1, \ f_3y, \ f_0y^{-1} = f_3y^2 + f_2y + f_1$
is an integral basis of $K(C)$ over $K[x]$ that reduces to an integral basis of $\FF_q[\overline{C}]$ over $\FF_q[x]$. Using the variable change 
$\mathsf{x} = x^{-1}$ and $\mathsf{y} = y / x^{e_2 - e_1}$ we find the patch 
\[ f^{\text{recipr.}}_3(\mathsf{x})\mathsf{y}^3 + f^{\text{recipr.}}_2(\mathsf{x})\mathsf{y}^2 + f^{\text{recipr.}}_1(\mathsf{x})\mathsf{y} + f^{\text{recipr.}}_0(\mathsf{x})
\]
above infinity, which admits an analogous integral basis. Here 
$f^{\text{recipr.}}_i$ denotes the degree 
$(i-1)e_1 + (2-i)e_2 + 2$ reciprocal of $f_i$. 
%Then, similarly,
%\begin{multline} \label{eq:intbasis_3_infty}
%    1, \quad f_3^{\text{recipr.}}(\mathsf{x})\mathsf{y} = f_3^{\text{recipr.}}(x^{-1}) \frac{y}{x^{e_2 - e_1}}, \quad 
%    f_3^{\text{recipr.}}(\mathsf{x})\mathsf{y}^2 + f_2^{\text{recipr.}}(\mathsf{x})\mathsf{y} + f_1^{\text{recipr.}}(\mathsf{x}) = \\ 
%    f_3^{\text{recipr.}}(x^{-1}) \left( \frac{y}{x^{e_2 - e_1}} \right)^2 + f_2^{\text{recipr.}}(x^{-1}) \frac{y}{x^{e_2 - e_1}} + f_1^{\text{recipr.}}(x^{-1})
%\end{multline}
%is an integral basis of $K(C)$ over $K[1/x]$ that reduces to an integral basis
%of $\FF_q[\overline{C}]$ over $\FF_q[1/x]$. 
We can supply these bases as additional input to Tuitman's algorithm, thereby by-passing Heuristic H.
\end{remark}

\section{Lifting curves in degree $d=4$} \label{sec:tetragonal}

\subsection*{Parametrizing quartic rings.} The parametrization of quartic $R$-rings $S$ is due to Bhargava \cite{bhargava4}. This time, the objects involved are pairs of ternary quadratic forms, up to an action of $\GL_3(R)\times \GL_2(R)$. For an element 
\[
(A, B)\in \GL_3(R)\times \GL_2(R),
\]
and a pair of ternary quadratic forms $(Q_1, Q_2)$ over $R$ represented as $3\times 3$ matrices, the action is defined by 
\[
(A, B) \ast (Q_1, Q_2) = B\cdot \begin{pmatrix} A Q_1 A^T \\ A Q_2 A^T \end{pmatrix}.
\]
Concretely, the quadratic forms associated with a quartic ring are obtained by specifying a \emph{cubic resolvent} (the next paragraph provides more details):

\begin{theorem}[Bhargava]
There is a canonical bijection between pairs $(S, S')$ where $S$ is a quartic ring over $R$ and $S'$ is a cubic resolvent for $S$, considered up to isomorphism, and pairs of ternary quadratic forms over $R$, up to the action of $\GL_3(R)\times \GL_2(R)$.
\end{theorem}

%This is an analogue of the classical cubic resolvent from Galois theory. 
See~\cite[Thm.\,1]{bhargava4}, although we will not explicitly rely on this theorem. But we will recycle its central map $\phi$, whose construction we briefly recall, while zooming in on our main case of interest, namely where $S$ is a domain, say with field of fractions $F$. We assume moreover that $F$ is a separable $S_4$-extension of $K = \ff R$, i.e., its Galois closure $E/K$ has as Galois group the full symmetric group $S_4$.
%\footnote{These restrictions can be dispensed with by working with the more general notion of \emph{$S_k$-closure}~\cite{bhargava4}.} 
Then a cubic resolvent for $S$ is a certain full-rank subring $S'\subseteq E^{D_4}=: F^\mathrm{res}$, where $D_4 = \langle (12), (1324)\rangle$, see~\cite[Def.\,8]{bhargava4} for a precise definition. In general, there might be more than one cubic resolvent ring, but for maximal rings it is unique~\cite[Cor.\,5]{bhargava4}. Note that if $F = K[y]/(f)$ with 
\[ f = (y - r_1)(y-r_2)(y-r_3)(y-r_4) = y^4 + ay^3 + by^2 + cy + d \]
then $F^\mathrm{res} = K[y] / (\res f)$ with
\begin{align*} 
\res f & = (y- r_1r_2 - r_3r_4)(y - r_1r_3 - r_2r_4)(y - r_1r_4 - r_2r_3) \\
& = y^3 - by^2 + (ac - 4d)y - (a^2d + c^2 - 4bd).
\end{align*}
%for which $\disc S'=\disc S$.  
This polynomial is famously known as \emph{Lagrange's cubic resolvent}. The most important feature of the Bhargava correspondence is the natural quadratic map
\begin{equation*} %\label{centralmap4}
\tilde\phi: F\to F^\mathrm{res}: \alpha\mapsto \alpha^{(1)}\alpha^{(2)} + \alpha^{(3)}\alpha^{(4)},
\end{equation*}
where the $\alpha^{(i)}$ denote the conjugates of $\alpha$ inside $E$ (numbered compatibly with the roots $r_i$). This map turns out to descend to a quadratic map of $R$-modules
\[
\phi: \frac{S}{R} \to \frac{S'}{R}.
\]
Upon taking bases for $S/R$ and $S'/R$ we obtain our two ternary quadratic forms over $R$. Changing bases of these modules then corresponds to an element of $\GL_3(R)\times \GL_2(R)$.
%Bhargava has proven that the above correspondence is a bijection.

%\paragraph{The resolvent curve.} Let $C\to \PP^1$ be a simply branched tetragonal curve over a field $k$. The field $k(C)^\mathrm{res}$ is a cubic extension of $k(t)$ and so corresponds to a unique smooth projective trigonal curve $C^\mathrm{res}\to \PP^1$. The idea to lift tetragonal curves is then similar to to trigonal curves. Take a reduced basis for $k(C)$ and for $k(C^\mathrm{res})$ and apply the above parametrization to get two quadratic forms. 

\subsection*{Lifting degree $4$ covers.} 
We can assume that $\overline{f}_4 = 1$, i.e., our input polynomial~\eqref{eq:inputpol} is monic.
Let $\FF_q(\overline{C})$ denote the function field it defines, which is a separable $S_4$-extension of $\FF_q(x)$ because $\overline{\varphi}$ is simply branched~\cite[Lem.\,6.10]{fulton}. Similarly, consider the cubic resolvent
\begin{equation}\label{eq: cubic resolvent polynomial} 
y^3 - \overline{f}_2y^2 + (\overline{f}_1 \overline{f}_3 - 4 \overline{f}_0)y - (\overline{f}_0 \overline{f}_3^2 + \overline{f}_1^2 - 4 \overline{f}_0 \overline{f}_2) 
\end{equation}
defining $\FF_q(\overline{C}^\mathrm{res}) := \FF_q(\overline{C})^\mathrm{res}$. We let $\FF_q[\overline{C}]_0$ and $\FF_q[\overline{C}^\mathrm{res}]_0$ be the respective integral closures of $R = \FF_q[x]$ inside these fields. It can be argued that $\FF_q[\overline{C}^\mathrm{res}]_0$ is the unique cubic resolvent ring $S'$ for $S = \FF_q[\overline{C}]_0$, but for our needs it suffices to know that $S' \subseteq \FF_q[\overline{C}^\mathrm{res}]_0$, which is immediate since $\FF_q[\overline{C}^\mathrm{res}]_0$ is maximal. 

Let $e_1, e_2, e_3$ be the Maroni invariants of $\overline{C}$ with respect to $\overline{\varphi}$, and let $b_1, b_2$ be its Schreyer invariants. 
%Let $\overline{C}\to \PP^1$ be a simply branched tetragonal curve over $\FF_q$. 
Take reduced $\FF_q[x]$-bases $1, \alpha_1, \alpha_2, \alpha_3 \in \FF_q[\overline{C}]_0$ and $1, \beta_1, \beta_2 \in \FF_q[\overline{C}^\mathrm{res}]_0$. With respect to these bases, the map $\phi$ above gives us two ternary quadratic forms $\overline{Q}_1, \overline{Q}_2 \in \FF_q[x][Y_1, Y_2, Y_3]$. 
To properly bound the degrees of their coefficients, we have to understand how the Maroni invariants of the resolvent curve $\overline{C}^\mathrm{res}$ relate to data associated with $\overline{C}$. Surprisingly, up to a small shift, these turn out to be the Schreyer invariants of $\overline{C}$ with respect to $\overline{\varphi}$:

\begin{theorem}\label{thm: maroni invariants of resolvent}
Let $k$ be a field of characteristic $\neq 2$ and consider a smooth projective curve over $k$ equipped with a simply branched degree $4$ morphism to $\PP^1$, say with Schreyer invariants $b_1, b_2$. Then the Maroni invariants of its cubic resolvent are $b_1 + 2, b_2 + 2$.
\end{theorem}
\begin{proof}
This result is due to Casnati~\cite[Def.\,6.4]{casnati}, although he formulated it in terms of Recillas' trigonal construction, which is the geometric counterpart of Lagrange's cubic resolvent, as pointed out in~\cite[\S8.6]{vangeemen}.
\end{proof}

%Dehomgenizing $A$ and $B$, they define two surfaces $Y$ and $Z$ in $\AA^3=\AA^1\times \AA^2$.

\begin{lemma} \label{lem:liftablequadrics}
The quadratic forms $\overline{Q}_1, \overline{Q}_2$ obtained through Bhargava's correspondence as above are a model of $\overline{C}$ of the form~\eqref{liftable4}.
%Let $Q_1, Q_2$ be any naive lifts of the forms $\overline{Q}_1, \overline{Q}_2$ constructed above. Then the affine curve $C$ in $\AA^3$ defined by $Q_1=Q_2=0$ is a lift of $\overline{C}$ of genus $g$.
%After homogenizing their coefficients, the quadratic forms $\overline{Q}_1, \overline{Q}_2$ become polynomials of the form~\eqref{eq:homdef} with $m = 2$ and $b = -b_1$ resp.\ $b = -b_2$, thereby defining a non-singular model of $\overline{C}$ in $\PP(\mathcal{O}_{\PP^1}(e_1) \oplus \mathcal{O}_{\PP^1}(e_2) \oplus \mathcal{O}_{\PP^1}(e_3))$.
%The dehomogenizations of $A$ and $B$ have Newton polygons as in figure \ref{fig:tet}. Hence $C$ is defined as in \todo{naar begin ergens referene}
\end{lemma}

%\begin{lemma}
%The intersection of $Y$ and $Z$ defines a model of $C$ in $\AA^3$.
%\end{lemma}
%\begin{proof}
%This follows from \cite[Sec.\,2]{bhargava5}.
%\end{proof}

\begin{proof}
Note that the polynomials indeed cut out a curve that is birationally equivalent with $\overline{C}$, in view of~\cite[\S2]{bhargava5}.\footnote{Alternatively, the reader can check that $\res_{y_2}(\overline{Q}_1'(y_1,y_2,1),\overline{Q}_2'(y_1,y_2,1)) = y_1^4 + \overline{f}_3y_1^3 + \overline{f}_2y_1^2 + \overline{f}_1y_1 + \overline{f}_0$, where $\overline{Q}_1'$ and $\overline{Q}_2'$ are the quadratic forms from below.} Since $1, \alpha_1, \alpha_2, \alpha_3$ and $1, \beta_1, \beta_2$ are reduced bases, 
%for $\FF_q[\overline{C}]_0$ resp.\ $\FF_q[\overline{C}^\mathrm{res}]_0$, 
by Theorem \ref{thm: maroni invariants of resolvent} we have that 
\begin{align*}
& 1, x^{-e_1-2}\alpha_1, x^{-e_2-2}\alpha_2, x^{-e_3-2}\alpha_3 \text{ and } \\
& 1, x^{-b_1-4}\beta_1, x^{-b_2-4}\beta_2
\end{align*}
are bases of $\FF_q[\overline{C}]_\infty$, resp.\ $\FF_q[\overline{C}^\mathrm{res}]_\infty$,
the integral closures of $\FF_q[x^{-1}]$ in 
$\FF_q(\overline{C})$, resp.\ $\FF_q(\overline{C}^\mathrm{res})$.
Now the quadratic map 
\[ \tilde\phi: \FF_q(\overline{C})\to \FF_q(\overline{C}^\mathrm{res}) \] 
from above also descends to a quadratic map of $\FF_q[x^{-1}]$-modules
\[
\phi': \frac{\FF_q[\overline{C}]_\infty}{\FF_q[x^{-1}]}\to \frac{\FF_q[\overline{C}^\mathrm{res}]_\infty}{\FF_q[x^{-1}]}.
\]
 With respect to the above bases, $\phi'$ is defined by two quadratic forms over $\FF_q[x^{-1}]$, which are necessarily obtained from $\overline{Q}_1$ and $\overline{Q}_2$ by applying the corresponding (diagonal) change of basis matrices. In other words, $\phi'$ is represented by the quadratic forms
\begin{align*}
    &x^{b_1+4}\overline{Q}_1(x^{-e_1-2}Y_1,x^{-e_2-2}Y_2, x^{-e_3-2}Y_3), \\ &x^{b_2+4}\overline{Q}_2(x^{-e_1-2}Y_1,x^{-e_2-2}Y_2, x^{-e_3-2}Y_3).
\end{align*}
But these have coefficients in $\FF_q[x^{-1}]$. Hence the degree of the $Y_iY_j$-coefficient in $\overline{Q}_1$ can be at most
$e_i+e_j-b_1$,
and similarly for $\overline{Q}_2$. In other words, the dehomogenized polynomials $\overline{Q}_1(y_1,y_2,1)$ and $\overline{Q}_2(y_1,y_2,1)$ are supported on the polytopes from Figure~\ref{fig:tet}. 
%Using that these polytopes have no interior $\ZZ^3$-points and that their Minkowski sum has $g$ interior $\ZZ^3$-points, a higher-dimensional version of Baker's theorem~\cite[Thm.\,1]{khovanskii} implies that the genus of $C$ is at most $g$. Combining this with the fact $\overline{C}$ is of genus $g$, we obtain the lemma.
\end{proof}

To compute these liftable quadrics $\overline{Q}_1, \overline{Q}_2$ in practice we will not directly compute the resolvent map $\phi$ with respect to reduced bases for $\FF_q(\overline{C})$ and $\FF_q(\overline{C}^\mathrm{res})$. Instead, we compute the map $\phi$ with respect to certain \emph{naive bases} for $\FF_q(\overline{C})$ and $\FF_q(\overline{C}^\mathrm{res})$ and then apply change of basis to a reduced basis.
In more detail, 
%we view $\FF_q(\overline{C})$ as the field of fractions of 
%\[
%\frac{\FF_q[x][y]}{(y^4+\overline{f}_3y^3+\overline{f}_2y^2 +\overline{f}_1y+\overline{f}_0)}
%\]
%and similarly for $\FF_q(\overline{C}^\mathrm{res})$. 
% (STOND HIERBOVEN AL GEDEFINIEERD)
denoting by $\overline{f}_i'$ the coefficients of the cubic resolvent polynomial of $\overline{f}$ as in \eqref{eq: cubic resolvent polynomial}, we consider the bases
\begin{align}\label{eq: naive bases tetragonal curve}
1, -\overline{f}_0y^{-1}, y, y^2 & \text{ for } \FF_q(\overline{C}) \text{ and} \\
1, y, -\overline{f}_0'y^{-1} & \text{ for } \FF_q(\overline{C}^\mathrm{res}). \nonumber
\end{align}
Computing the representation of the resolvent map $\phi$ with respect to these bases can be done symbolically by means of Vieta's formulas, yielding the quadrics
\begin{equation} \label{vietaquadric}
\overline{Q}_1' = \begin{pmatrix}
\overline{f}_0 & 0 & \frac{\overline{f}_1}{2} \\
0 & 1 & \frac{-\overline{f}_3}{2} \\
\frac{\overline{f}_1}{2} & \frac{-\overline{f}_3}{2} & \overline{f}_2
\end{pmatrix}, \qquad \overline{Q}_2' = \begin{pmatrix}
0 & \frac{-1}{2} & \frac{\overline{f}_3}{2} \\
\frac{-1}{2} & 0 & 0 \\
\frac{\overline{f}_3}{2} & 0 & 1
\end{pmatrix}.
\end{equation}
Now let $1, \alpha_1, \alpha_2, \alpha_3$ and $1, \beta_1, \beta_2$ be reduced bases for $\FF_q[\overline{C}]_0$, resp.\ $\FF_q[\overline{C}^\mathrm{res}]_0$, as above. To compute the cubic resolvent map with respect to these bases, we simply apply the change of basis action from the naive bases in \eqref{eq: naive bases tetragonal curve} to these reduced bases. 
We note that this involves elements of $\GL_3(\FF_q(x))\times \GL_2(\FF_q(x))$ rather than
$\GL_3(\FF_q[x])\times \GL_2(\FF_q[x])$.
The resulting quadrics $\overline{Q}_1, \overline{Q}_2$ will be our model of the form~\eqref{liftable4}. Then, as explained in Section~\ref{sec:preliminaries}, we can take any $Q_1, Q_2\in \mathcal{O}_K[x][y_1, y_2]$ lifting the $\overline{Q}_i(y_1,y_2,1)$'s in a support-preserving way. In order to find a plane model, we can compute the resultant $\res_{y_2}(Q_1, Q_2)$,
which is indeed of degree $4$ in $y=y_1$. After making it monic, it can be fed as input to Tuitman's algorithm.

\section{Lifting curves in degree $d=5$} \label{sec:pentagonal}

\subsection*{Parametrizing quintic rings.} The parametrization of quintic $R$-rings $S$ is also due to Bhargava \cite{bhargava5}. We assume that $\charac R\neq 2,3$. The objects involved in the parametrization are now quadruples of $5\times 5$ skew-symmetric matrices over $R$. There is a natural action of $\GL_5(R) \times \GL_4(R) $ on such objects, given by
\[
(A,B)\ast M = B\cdot 
\begin{pmatrix} A M_1 A^T \\
A M_2 A^T \\
A M_3 A^T \\
A M_4 A^T 
\end{pmatrix},
\]
with $M=(M_1, M_2, M_3, M_4)$ a quadruple of $5\times 5$ skew-symmetric matrices and $(A, B)\in \GL_5(R) \times \GL_4(R)$. Here too, the parametrization requires us to specify a \emph{sextic resolvent}  (see the next paragraph for details):

\begin{theorem}[Bhargava]
There is a canonical bijection between pairs $(S,S')$ where $S$ is a quintic ring and $S'$ is a sextic resolvent for $S$, considered up to isomorphism, and quadruples of $5\times 5$ skew-symmetric matrices over $R$, up to the action of $\GL_5(R) \times \GL_4(R)$.
\end{theorem}

See~\cite{bhargava5}, although as in the previous sections, we will not explicitly rely on this theorem. But we will need the fundamental resolvent map~\eqref{fundamentalmap} below. Let us again focus on the setting where $S$ is a domain with field of fractions $F$, and let $K = \ff R$. We assume that $F$ is a separable $S_5$-extension of $K$, i.e., its Galois closure $E/K$ has as Galois group the whole of $S_5$. Consider the order $20$ subgroup $H=H^{(1)}=\operatorname{AGL}_1(\FF_5)=\langle (12345), (1243)\rangle \subseteq S_5$. 
Then a sextic resolvent for $S$ is a certain full-rank subring $S' \subseteq E^H =: F^\mathrm{res}$; for a precise definition we refer to~\cite[Def.\,5]{bhargava5}. In general, such a sextic resolvent ring is not unique, but for maximal quintic rings it is~\cite[Cor.\,19]{bhargava5}. If $F = K[y]/(f)$ with
\[ f = (y-r_1)(y-r_2)(y-r_3)(y-r_4)(y-r_5) = y^5 + ay^4 + by^3 + cy^2 + dy + e,\]
then $F^\mathrm{res} = K[y]/(\res f)$ with
$\res f = (y - \rho_1)(y - \rho_2)(y - \rho_3)(y - \rho_4)(y - \rho_5)(y - \rho_6)$,
where 
\[ \rho_1 = (r_1r_2 + r_2r_3 + r_3r_4 + r_4r_5 + r_5r_1 - r_1r_3 - r_3r_5 - r_5r_2 - r_2r_4 - r_4r_1)^2 \]
and $\{\rho_1, \rho_2, \ldots, \rho_6\}$ is the orbit of $\rho_1$ under the natural $S_5$-action permuting the $r_i$'s.
Note that $\rho_1$ is stabilized by $H^{(1)}$. We choose
$\rho_{2+i}$ to be stabilized by the conjugate subgroup
\[
H^{(2+i)} = (12345)^{-i}\langle (13254), (3245)\rangle (12345)^{i}, \text{ for } 0\leq i\leq 4.
\]
The polynomial $\res f$ is known as \emph{Cayley's sextic resolvent}; concrete expressions for its coefficients in terms of $a,b,c,d,e$ can be found in~\cite[Proof of Prop.\,13.2.5]{cox}.\footnote{Or it can be found hard-coded in our accompanying Magma file \texttt{precomputed\_{}5.m}.}

%The corresponding resolvent field $F^H$ will be denoted by $K^\mathrm{res}$, this resolvent is also known as \emph{Cayley's sextic resolvent}. This is a sextic field over $\ff R$. We consider also the conjugate subgroups

For an element $\alpha\in F^\mathrm{res}$ we denote by $\alpha^{(i)}$ the conjugates of $\alpha$ inside $E$, labeled so that $\alpha^{(i)}$ is fixed by $H^{(i)}$. Consider bases $\alpha_0 = 1, \alpha_1, \ldots, \alpha_4$ for $S / R$ and $\beta_0 = 1, \beta_1, \ldots, \beta_5$ for $S' / R$, and define
\[ \sqrt{\disc S} = \begin{vmatrix} 
1 & 1 & \ldots & 1 \\ 
\alpha_1^{(1)} & \alpha_1^{(2)} & \ldots & \alpha_1^{(5)} \\ 
\vdots & \vdots & \ddots & \vdots \\
\alpha_4^{(1)} & \alpha_4^{(2)} & \ldots & \alpha_4^{(5)}  
\end{vmatrix}. 
%\qquad \sqrt{\disc S'} = \begin{vmatrix} 
%1 & 1 & \ldots & 1 \\ 
%\beta_1^{(1)} & \beta_1^{(2)} & \ldots & \beta_1^{(6)} \\ 
%\vdots & \vdots & \ddots & \vdots \\
%\beta_5^{(1)} & \beta_5^{(2)} & \ldots & \beta_5^{(6)}  
%\end{vmatrix} 
\]
%and let $\disc S = (\sqrt{ \disc S})^2$.
The central tool in Bhargava's correspondence is the \emph{fundamental resolvent map}, which is the bilinear alternating form
\begin{equation} \label{fundamentalmap}
g : F^\mathrm{res} \times F^\mathrm{res} \to F : (\alpha, \beta)\mapsto \sqrt{\disc S} \cdot \begin{vmatrix} 1 & 1 & 1 \\
\alpha^{(1)}+\alpha^{(2)} & \alpha^{(3)} + \alpha^{(6)} & \alpha^{(4)}+\alpha^{(5)} \\
\beta^{(1)}+\beta^{(2)} & \beta^{(3)} + \beta^{(6)} & \beta^{(4)}+\beta^{(5)}\end{vmatrix}.
\end{equation}
This turns out to descend to a well-defined map 
$\tilde{S'}\times \tilde{S'}\to \tilde{S}$, where
\[
\tilde{S} = R\alpha_1^* + R\alpha_2^* + R\alpha_3^* + R\alpha_4^*\subseteq F, \qquad \tilde{S}' = R\beta_1^* + R\beta_2^* + R\beta_3^* + R\beta_4^* + R\beta_5^* \subseteq F^\mathrm{res}
\]
are defined in terms of the dual bases $\alpha_0^*, \ldots , \alpha_4^*$ and $\beta_0^*, \ldots , \beta_5^*$ with respect to the trace pairing, i.e., $\tr_{F/K} ( \alpha_i\alpha_j^*) = \delta_{ij}$ (with $\delta_{ij}$ the Kronecker delta), and similarly for $\beta_j^*$.
Note that the extensions $F/K$ and $F^\mathrm{res}/K$ are both separable and so their trace pairings are non-degenerate.
%So when fixing a basis , there exists a unique basis $\alpha_0^*, \ldots , \alpha_4^*$ of $F / K$ for which $\tr_{F/K} ( \alpha_i\alpha_j^*) = \delta_{ij}$
%(with $\delta_{ij}$ the Kronecker delta), which we use to define the lattice
%\[
%\tilde{S} = R\alpha_1^* + R\alpha_2^* + R\alpha_3^* + R\alpha_4^*\subseteq F 
%\]
%A \emph{resolvent ring} for $S$ will be a certain sextic ring $S'$ in $K^\mathrm{res}$, see \cite[Def.\, 5]{bhargava5}. Any quintic ring has a sextic resolvent, but it is not unique in general. 
%We similarly fix a basis $\beta_0 = 1, \beta_1, \ldots, \beta_5$ for $S' / R$, along with its dual $\beta_0^*, \ldots, \beta_5^*$ and the five-dimensional lattice
%\[
%\tilde{S}' = R\beta_1^* + R\beta_2^* + R\beta_3^* + R\beta_4^* + R\beta_5^* \subseteq E^H.
%\]
%The fundamental ingredient of Bhargava's correspondence is then the bilinear alternating form \todo{check this more carefully, bij bhargava is er nog een factor 4/3: geen probleem in char 3?}
%\[
%g: \tilde{S'}\times \tilde{S'}\to \tilde{S}: (\alpha, \beta)\mapsto \sqrt{\disc S'}\begin{vmatrix} 1 & 1 & 1 \\
%\alpha^{(1)}+\alpha^{(2)} & \alpha^{(3)} + \alpha^{(6)} & \alpha^{(4)}+\alpha^{(5)} \\
%\beta^{(1)}+\beta^{(2)} & \beta^{(3)} + \beta^{(6)} & \beta^{(4)}+\beta^{(5)}\end{vmatrix}
%\]
%where 
%\[ \sqrt{\disc S'} = \det \begin{pmatrix} 
%1 & 1 & \ldots & 1 \\ 
%\beta_1^{(1)} & \beta_1^{(2)} & \ldots & \beta_1^{(6)} \\ 
%\vdots & \vdots & \ddots & \vdots \\
%\beta_5^{(1)} & \beta_5^{(2)} & \ldots & \beta_5^{(6)}  
%\end{pmatrix}. \]
With respect to the bases $\{ \beta_i^*\}_i$ and $\{ \alpha_i^* \}_i$, the map
$g$ is represented by a quadruple $M = (M_1, M_2, M_3, M_4)$ of $5\times 5$ skew-symmetric matrices. Changing bases of $\tilde{S}'$ and $\tilde{S}$ then corresponds to an element of $\GL_5(R)\times \GL_4(R)$. 
%Bhargava has proven that this correspondence is a bijection.

\begin{remark}
Our fundamental resolvent map differs from Bhargava's original map by a factor $4/3$, which is not an issue in view of our restrictions on the field characteristic.
\end{remark}

\subsection*{Lifting degree $5$ covers.} 
As in the $d=4$ case, we assume that our input polynomial $\overline{f}$ from~\eqref{eq:inputpol} is monic (i.e., $\overline{f}_5 = 1$). 
Let $\FF_q(\overline{C})$ be the corresponding function field; this is a separable $S_5$-extension of $\FF_q(x)$ because $\overline{\varphi}$ is simply branched~\cite[Lem.\,6.10]{fulton}.
We also consider Cayley's sextic resolvent associated with our input polynomial, defining $\FF_q(\overline{C}^\mathrm{res}) := \FF_q(\overline{C})^\mathrm{res}$. Let $\FF_q[\overline{C}]_0$ and $\FF_q[\overline{C}^\mathrm{res}]_0$ be the respective integral closures of $R = \FF_q[x]$ inside these two function fields; it can be argued that $\FF_q[\overline{C}^\mathrm{res}]_0$ is the unique sextic resolvent ring $S'$ for $S = \FF_q[\overline{C}]_0$, but as in the $d=4$ case it suffices to observe that $S' \subseteq \FF_q[\overline{C}^\mathrm{res}]_0$. 

Let $e_1, e_2, e_3, e_4$ be the Maroni invariants of $\overline{C}$ with respect to $\overline{\varphi}$, and let $b_1, b_2, b_3, b_4, b_5$ be its Schreyer invariants.
Take reduced $\FF_q[x]$-bases $1, \alpha_1, \ldots, \alpha_4 \in \FF_q[\overline{C}]_0$ and $1, \beta_1, \ldots, \beta_5 \in \FF_q[\overline{C}^\mathrm{res}]_0$ and consider the quadruple $(\overline{M}_1, \overline{M}_2, \overline{M}_3, \overline{M}_4)$ of $5\times 5$ skew-symmetric matrices over $\FF_q[x]$ arising along the above construction.
%Let $C\to \PP^1$ be a simply branched pentagonal curve over a field $k$ of characteristic not $2$ or $3$. The field $k(C)^\mathrm{res}$ is a sextic extension of $k(x)$ and so corresponds to a unique smooth projective curve, which we denote by $C^\mathrm{res}$ and call the \emph{resolvent curve}. This curve naturally comes equipped with a degree $6$ map $C^\mathrm{res}\to \PP^1$ coming from the embedding $k(x)\subseteq k(C^\mathrm{res})$. Taking reduced bases for both $k(C)$ and $k(C^\mathrm{res})$ the form $g$ is represented by a quadruple $(A_1, A_2, A_3, A_4)$ of $5\times 5$ skew-symmetric matrices over $k[x]$. 
We represent this by the single matrix 
\[
\overline{M} = \overline{M}_1Y_1+\overline{M}_2Y_2 + \overline{M}_3Y_3 + \overline{M}_4Y_4\in k[x][Y_1, Y_2, Y_3, Y_4]
\]
whose entries are now linear and homogeneous in the $Y_i$. 
To get a handle on the degrees of their coefficients, we should again express the Maroni invariants of the resolvent curve $\overline{C}^\mathrm{res}$ in terms of data associated with $\overline{C}$. As in the case of the cubic resolvent, this can be done in a surprisingly explicit way:
\begin{theorem}\label{thm: maroni invariants of sextic resolvent}
Let $k$ be a field of characteristic $\neq 2$ and consider a smooth projective curve over $k$ equipped with a simply branched degree $5$ morphism to $\PP^1$, say with Schreyer invariants $b_1, \ldots, b_5$. Then the Maroni invariants of its sextic resolvent are $g-2-b_5, \ldots , g-2-b_1$.
\end{theorem}
\begin{proof}
This theorem seems new and is part of a generalization of 
Theorem~\ref{thm: maroni invariants of resolvent}, which is currently being elaborated in collaboration with Yongqiang Zhao~\cite{resolvent_project}. In the meantime, a proof of Theorem~\ref{thm: maroni invariants of sextic resolvent} can be found in the master thesis of the second listed author~\cite{vermeulen}.
\end{proof}

%The idea is that this is a model for $C$ as described in \todo{refereere naar intro ergens}. To make this precise, we need to have an analogue of Casnati's theorem, describing the Maroni invariants of the resolvent curve $C^\mathrm{res}$.

%\paragraph{Lifting pentagonal curves.} We now specialize the above to $k=\FF_q$. So let $\overline{C}\to \PP^1$ be a simply branched pentagonal curve over $\FF_q$. Take reduced bases $1, \alpha_1, ..., \alpha_4$ for $\FF_q[\overline{C}]_0$ and $1, \beta_1, ..., \beta_5$ for $\FF_q[\overline{C}^\mathrm{res}]_0$. With respect to the dual bases $\alpha_1^*, ..., \alpha_4^*$ and $\beta_1^*, ..., \beta_5^*$ the form $g$ is represented by a $5\times 5$ skew-symmetric matrix $A$ whose entries are polynomials in $k[x][y_1, y_2, y_3, y_4]$ linear and homogeneous in the $y_i$. We consider instead the dehomogenization to $y_4=1$ of $A$.

\begin{lemma}
%After homogenizing its coefficients, the entry of $\overline{M}$ in row $i$ and column $j \neq i$ is a polynomial of the form~\eqref{eq:homdef} with $m=1$ and $b = - (b_i + b_j + 6 - g)$. In particular, the five $4 \times 4$ sub-Pfaffians define a non-singular model of $\overline{C}$ in $\PP(\mathcal{O}_{\PP^1}(e_1) \oplus \ldots \oplus  \mathcal{O}_{\PP^1}(e_4))$.
Denote by $\overline{M}_{r,i,j}$ the $(i,j)$-th entry of the matrix $\overline{M}_r$ constructed through Bhargava's correspondence as above. Then $\deg \overline{M}_{r,i,j} \leq e_r+b_i+b_j +6-g$. In particular, this defines a model for $\overline{C}$ of the form~\eqref{liftable5}.
%Let $M$ be a naive skew-symmetric lift of $\overline{M}$ to $\mathcal{O}_K$. Then the five $4\times 4$ sub-Pfaffians of $M$ define a curve $C$ inside $\AA^4$ lifting $\overline{C}$ and of genus $g$.
\end{lemma}

\begin{proof}
The fact that the sub-Pfaffians of $\overline{M}$ cut out a curve birational to $\overline{C}$ follows again from~\cite[\S2]{bhargava5}. 
%\todo{mss nog wat detailleren nadat we geval kar $3$ beter begrijpen}
As for the claim on the degrees, we apply the same proof strategy as in the degree $4$ case. Denote by $\FF_q[\overline{C}]_\infty$ the integral closure of $\FF_q[x^{-1}]$ in $\FF_q(\overline{C})$. Let $g_0$ be the fundamental resolvent form attached to the basis $1, \alpha_1, \ldots, \alpha_4$ of $\FF_q[\overline{C}]_0$ over $\FF_q[x]$, and let $g_\infty$ be the fundamental resolvent form attached to the basis $1, x^{-e_1-2}\alpha_1, \ldots, x^{-e_4-2}\alpha_4$ of $\FF_q[\overline{C}]_\infty$ over $\FF_q[x^{-1}]$. 
%Note that we can naturally extend both these maps to bilinear alternating functions $\FF_q(\overline{C}^\mathrm{res})\times \FF_q(\overline{C}^\mathrm{res})\to \FF_q(\overline{C})$. 
We have that, for all $u,v\in \FF_q(\overline{C}^\mathrm{res})$,
\[
g_0(u,v) = \frac{\sqrt{\disc \FF_q[\overline{C}]_0}}{\sqrt{\disc \FF_q[\overline{C}]_\infty}} g_\infty(u,v) = x^{g+4}g_\infty(u,v).
\]
%A straightforward computation using reduced bases shows that $\sqrt{\disc \FF_q[\overline{C}]_0 / \disc \FF_q[\overline{C}]_\infty}$ is nothing but $x^{g+4}$.
Let $\alpha_0^*, \ldots, \alpha_4^*$, resp.\ $\beta_0^*, \ldots, \beta_5^*$, be dual bases for $1, \alpha_1, \ldots, \alpha_4$, resp.\ $1, \beta_1, \ldots,\beta_5$. Then the corresponding dual bases for the rings $\FF_q[\overline{C}]_\infty$ and $\FF_q[\overline{C}^\mathrm{res}]_\infty$ are
\begin{align*}
\alpha_0^*, x^{e_1+2}\alpha_1^*, \ldots, x^{e_4+2}\alpha_4^* \text{ for } \FF_q[\overline{C}]_\infty, \\
\beta_0^*, x^{e_1'+2}\beta_1^*, \ldots, x^{e_5'+2}\beta_5^* \text{ for } \FF_q[\overline{C}^\mathrm{res}]_\infty,
\end{align*}
where the $e_i'$ are the Maroni invariants of the resolvent. We now compute, for $i,j>0$,
\begin{align}
g_\infty(x^{e_i'+2}\beta_i^*, x^{e_j+2}\beta_j^*) &= x^{e_i'+e_j'+4}x^{-g-4}g_0(\beta_i^*, \beta_j^*) \\ 
&= \sum_{l=1}^4 x^{-e_l-g-2+e_i'+e_j'}(\overline{M}_l)_{ij}(x^{e_l+2}\alpha_l^*).
\end{align}
It follows that $g_\infty$ is represented by the matrix whose entries have coefficients 
\[
x^{-e_l-g-2+e_i'+e_j'}(\overline{M}_l)_{ij},\quad i,j=1, \ldots, 5, \ l=1, \ldots, 4.
\]
But these coefficients belong to $\FF_q[x^{-1}]$. Hence we find that
$\deg (\overline{M}_l)_{ij} \leq e_l+b_i+b_j+6 - g$
by Theorem~\ref{thm: maroni invariants of sextic resolvent}, as wanted.
%We are not aware of any analogue of Baker's bound applicable in this situation to conclude that $C$ has genus $g$. \todo{volgt ook via schreyer? betere verwijzing} Instead one can rely on Bhargava's work~\cite{bhargava5} in combination with the Riemann--Hurwitz formula.
\end{proof}

To compute such a liftable matrix in practice, we follow a similar approach as in the case of degree $4$ covers. Namely, we will not be computing the fundamental resolvent map with respect to our reduced bases directly, but rather compute this for certain naive bases and apply change of basis. 
%Suppose that our input polynomial is monic. There is an explicit polynomial defining the sextic resolvent field, and we denote by $\overline{f}_i'$ its coefficients. 
%As before, we consider $\FF_q(\overline{C})$ as the field of fractions of 
%\[
%\frac{\FF_q[x][y]}{(y^5+ \overline{f}_4y^4+\overline{f}_3y^3+\overline{f}_2y^2 +\overline{f}_1y+\overline{f}_0)}
%\]
%and similarly for $\FF_q(\overline{C}^\mathrm{res})$. 
Concretely, consider the naive bases
\begin{align*}
1, y, y^2, y^3, y^4 & \text{ for } \FF_q(\overline{C}), \text{ and }\\
1, y, y^2, y^3, y^4, y^5 & \text{ for } \FF_q(\overline{C}^\mathrm{res}),
\end{align*}
along with the slightly altered fundamental resolvent map 
\[
g': \FF_q(\overline{C}^\mathrm{res})\times \FF_q(\overline{C}^\mathrm{res})\to \FF_q(\overline{C}): (\alpha, \beta)\mapsto \sqrt{\disc \overline{f}} \cdot \begin{vmatrix} 1 & 1 & 1 \\
\alpha^{(1)}+\alpha^{(2)} & \alpha^{(3)} + \alpha^{(6)} & \alpha^{(4)}+\alpha^{(5)} \\
\beta^{(1)}+\beta^{(2)} & \beta^{(3)} + \beta^{(6)} & \beta^{(4)}+\beta^{(5)}\end{vmatrix}
\]
where $\sqrt{\disc \overline{f}} = \det ((y^i)^{(j)})_{0 \leq i \leq 4, 1 \leq j \leq 5}$.
%Note that this resolvent map makes sense on all of $\FF_q(\overline{C}^\mathrm{res})$ and not just on $\overline{S}'$ as defined above. 
We compute the $\overline{M}_{ij}'^{(r)}\in \FF_q[x]$ for which
\[
g'(y^i, y^j) = \sum_{r=0}^4 \overline{M}'^{(r)}_{ij}y^r,
\]
giving five $5\times 5$ skew-symmetric matrices $\overline{M}'^{(0)}, \ldots, \overline{M}'^{(4)}$; here we used that $\overline{M}'^{(r)}_{ij}=0$ as soon as $i$ or $j$ is zero, allowing us to disregard these terms. We call this the \emph{naive model}.

\begin{remark} It is important to note that these expressions can be computed symbolically in terms of the coefficients $\overline{f}_i$ of $\overline{f}$, by means of Vieta's formulas. Therefore this computation only has to be done once for all curves. This is in complete analogy with the degree $4$ case, see~\eqref{vietaquadric}. However, there the naive model was very simple, whereas this time the expressions involved are rather long. However, a computer has no trouble with these computations.
%The reason for working with $g'$ instead of $g$ is that this expression can be computed symbolically in terms of the coefficients of $f$.
\end{remark}
Now compute reduced bases
$1, \alpha_1, \ldots, \alpha_4$ for $\FF_q[\overline{C}]_0$ and
$1, \beta_1, \ldots, \beta_5$ for $\FF_q[\overline{C}^\mathrm{res}]_0$
along with their corresponding dual bases. Acting on the naive model with a change of basis from the naive bases to the duals of these reduced bases, yields the altered resolvent map $g'$ with respect to these dual reduced bases. Note that this action will be by an element of $\GL_5(\FF_q(x))\times \GL_4(\FF_q(x))$ rather than $\GL_5(\FF_q[x])\times \GL_4(\FF_q[x])$. To obtain instead the resolvent map $g$ we have to multiply by 
\[
\frac{\sqrt{\disc \FF_q[\overline{C}]_0}}{\sqrt{\disc \overline{f}}}.
\]
Since we already have the reduced bases at hand, this factor is easiest to compute as the determinant of the change of basis matrix from the naive basis for $\FF_q(\overline{C})$ to the reduced basis $1, \alpha_1, \ldots, \alpha_4$.

At this point, we have a representation of the fundamental resolvent map $g$ with respect to the duals of the reduced bases for $\FF_q[\overline{C}]_0$ and $\FF_q[\overline{C}^\mathrm{res}]_0$ as a $5\times 5$ skew-symmetric matrix $\overline{M}$ with entries in $k[x][Y_1, Y_2, Y_3, Y_4]$, linear and homogeneous in the $Y_i$. This is the desired model, which we can lift naively, in a skew-symmetry preserving way, to a matrix having entries in $\mathcal{O}_K[x][Y_1, Y_2, Y_3, Y_4]$. Computing its five $4 \times 4$ sub-Pfaffians, dehomogenizing, and then eliminating variables finally returns our output~\eqref{eq:liftedpol}, ready to be fed as input to Tuitman's algorithm.

\small

\vspace{0.1cm}

\noindent \textsc{Cosic, research group at imec and KU Leuven}

\noindent \textsc{Kasteelpark Arenberg 10/2452, 3001 Leuven (Heverlee), Belgium}

\noindent \vspace{-0.35cm}

\noindent \textsc{Department of Mathematics: Algebra and Geometry, Ghent University}

\noindent \textsc{Krijgslaan 281 -- S25, 9000 Gent, Belgium}

\vspace{0.2cm}

\noindent \textsc{Section of Algebra, Department of Mathematics, KU Leuven}

\noindent \textsc{Celestijnenlaan 200B, 3001 Leuven (Heverlee), Belgium}

\noindent \vspace{-0.2cm}

\noindent \texttt{firstname.lastname@kuleuven.be}

\appendix

\section{Magma implementation: discussion and examples} \label{app:report}

The approximate timings mentioned below were obtained using Magma V2.25-2 on \texttt{kraitchik}, a computer with $12$ Intel Xeon E5-2630 v2 processors and $128$GB of memory, running Ubuntu 16.04.

%For our own convenience, the timings in this section were obtained using Tuit\-man's implementation from~\cite{tuitman2}, which starts off by computing its own integral bases through the Magma intrinsic, hoping that these meet the algorithm's requirements (which they usually do). By instead using the integral bases coming from our lifting step, this computation can in fact be skipped. This would not lead to a significant speed-up, but it would get rid of potential failures.

\subsection*{Degree $d=3$.} In the accompanying Magma file \texttt{lifting\_{}lowgonal\_{}3.m}, the user can choose a finite field $\FF_q$ of characteristic $p > 2$, along with a suitable pair of integers $e_1$ and $e_2$. Running the code 
\begin{itemize} 
\item first generates a random degree $3$ cover $\overline{C} \to \PP^1$ over $\FF_q$ whose Maroni invariants are $e_1, e_2$, of which it chooses a somewhat scrambled defining polynomial having the form~\eqref{eq:inputpol}; this serves as test input for our lifting procedure,
\item next applies the Delone--Faddeev correspondence to this input, thereby procuding a naively liftable defining polynomial, as discussed in Section~\ref{sec:trigonal},
\item finally carries out the naive lift and, after making the result monic, prints it to a file \texttt{inputcurve\_{}3.m}, which can be loaded as input to Tuitman's \texttt{pcc\_{}p.m} or \texttt{pcc\_{}q.m} implementation.
\end{itemize}
E.g., over $\FF_{11}$, a run of our code generated the random trigonal curve
\begin{multline*}
\scriptstyle
 (6x^2 + 7)y^3 + (7x^{12} + 10x^{10} + 2x^4 + 3x^3 + 8x^2 + 7x + 1)y^2 + \\
 \scriptstyle
    (7x^{22} + 10x^{20} + 4x^{14} + 6x^{13} + 5x^{12} + 3x^{11} + 2x^{10} + 4x^6 +
    x^5 + x^4 + 8x^3 + 8x + 3)y + \\
\scriptstyle
    6x^{32} + 7x^{30} + 2x^{24} + 3x^{23} +
    8x^{22} + 7x^{21} + x^{20} + 4x^{16} + x^{15} + x^{14} + 8x^{13} \\ 
\scriptstyle + 8x^{11} + 3x^{10} 
    + 10x^8 + 4x^7 + 7x^6 + 2x^5 + 6x^4 + x^3 + 9x^2 + 3x + 5 = 0
\end{multline*}
of genus $8$, having prescribed Maroni invariants $\{2, 4\}$. Under the Delone--Faddeev correspondence this was transformed into
\begin{multline*}
\scriptstyle    (10x^2 + 8)y^3 + (8x^4 + x^3 + 10x^2 + 7x + 1)y^2 + (9x^6 + 5x^5 +
    3x^4 + 2x^3 + 4x^2 + 7x + 9)y \\ 
\scriptstyle    + x^8 + 4x^7 + 5x^6 + x^5 + 4x^4 +
    9x^3 + 6x^2 + 9 = 0.
\end{multline*}
After taking a naive lift having coefficients in $\{-5, \ldots, 5\} \subseteq \ZZ$ and making the result monic in $y$, 
this was fed to Tuitman's code, which determined the numerator of the Hasse--Weil zeta function as
\begin{multline*}
\scriptstyle
214358881T^{16} - 38974342T^{15} + 30116537T^{14} - 4509428T^{13} + 2459688T^{12}
    - 505780T^{11} + 151855T^{10} \\ 
\scriptstyle - 59070T^9 + 8366T^8 - 5370T^7 + 1255T^6
    - 380T^5 + 168T^4 - 28T^3 + 17T^2 - 2T + 1.
\end{multline*}
On a larger scale, for a random trigonal genus $9$ curve over $\FF_{5^9}$ having Maroni invariants $\{3, 4\}$, the same procedure computed its Hasse--Weil zeta function in about $20$ minutes. For a random trigonal genus $8$ curve over $\FF_{7^{12}}$ having Maroni invariants $\{3,3\}$ we obtained its Hasse--Weil zeta function using roughly $2$ hours of computation. In both cases, the lifting step took less than $0.1$ seconds.

\subsection*{Degree $d=4$.} In the accompanying Magma file \texttt{lifting\_{}lowgonal\_{}4.m}, the user chooses a finite field $\FF_q$ of characteristic $p > 2$, along with a suitable quintuple of integers $e_1, e_2, e_3$, $b_1, b_2$. Running the code 
\begin{itemize} 
\item first generates a random degree $4$ cover $\overline{C} \to \PP^1$ over $\FF_q$ with Maroni invariants $e_1, e_2, e_3$ and Schreyer invariants $b_1, b_2$, of which it chooses a somewhat scrambled monic defining polynomial; this serves as test input for our lifting procedure,
\item next applies the Bhargava correspondence to this input, thereby procuding a naively liftable pair of quadratic forms (i.e., symmetric matrices in $\FF_q[x]^{3 \times 3}$), as discussed in Section~\ref{sec:tetragonal},
\item finally carries out the naive lift and, after taking a resultant and making the outcome monic, prints it to a file \texttt{inputcurve\_{}4.m}, which can be loaded as input to Tuitman's \texttt{pcc\_{}p.m} or \texttt{pcc\_{}q.m} implementation.
\end{itemize}
E.g., over $\FF_{7}$ a run of our code generated the random tetragonal curve
\begin{multline*}
\scriptstyle y^4 + (4x^{10} + 6x^4 + 2x^3 + 3x^2 + 5x + 6)y^3 \\ 
\scriptstyle + (6x^{20} + 4x^{14} +
    6x^{13} + 2x^{12} + x^{11} + 4x^{10} + 4x^8 + 2x^6 + x^4 + x^3 + 6x^2)y^2 \\
\scriptstyle    + (4x^{30} + 4x^{24} + 6x^{23}  + 2x^{22} + x^{21} + 4x^{20} + x^{18} + 4x^{16} +
    2x^{14} + 2x^{13} \\ 
\scriptstyle + 5x^{12} + 6x^{11} + 4x^9 + 6x^7 + x^6 + 5x^5 + 6x^4 +
    6x^3 + 2x^2 + 5x + 5)y \\ 
\scriptstyle + x^{40} + 6x^{34} + 2x^{33} + 3x^{32} + 5x^{31} +
    6x^{30} + 4x^{28} + 2x^{26} + x^{24} + x^{23} + 6x^{22} + 6x^{21} + 4x^{19} +
    6x^{17} + x^{16} \\ 
\scriptstyle + 5x^{15} + 6x^{14} + 3x^{13} + 5x^{12} + 4x^{11} + x^{10} + 2x^9
    + x^8 + 5x^7 + x^6 + 3x^5 + 2x^4 + x^3 + 3x^2 + x + 2
\end{multline*}
of genus $10$, having Maroni invariants $\{1, 2, 4\}$ and Schreyer invariants $\{2,3\}$. Bhargava's correspondence then produced the pair of matrices
\begin{multline*}
 \left( \begin{smallmatrix} 5 &  4x + 4 &  2x^3 + 5x^2 + 4x + 5 \\
4x + 4  & 2x^2 + 6x + 2 &  6x^3 + 6x^2 + 5 \\
2x^3 + 5x^2 + 4x + 5 &  6x^3 + 6x^2 + 5 &  5x^6 + 3x^5 + 6x^4 + 2x^3
    + 2x + 1 \end{smallmatrix} \right),  \\
    \scriptstyle
\left( \begin{smallmatrix} 0 &  0 &  5x^2 + 3x + 2 \\
   0  & 1 &  2x^3 + x^2 + 6x + 4 \\
   5x^2 + 3x + 2 &  2x^3 + x^2 + 6x + 4  & 3x^5 + 6x^4 + 4x^3 + 6x^2 + 4 \\ 
\end{smallmatrix} \right).
\end{multline*}
These matrices were then lifted naively to characteristic zero, i.e., to matrices over $\ZZ[x]$ whose entries have coefficients in $\{-3, \ldots, 3\}$. After taking a resultant of the corresponding (dehomogenized) quadratic forms and making the result monic, 
we obtained a polynomial of the form~\eqref{liftable4} which was fed as input to Tuitman's code. The numerator of the Hasse--Weil zeta function was then determined as
\begin{multline*}
\scriptstyle  282475249T^{20} + 161414428T^{19} + 80707214T^{18} + 24706290T^{17} +
    5764801T^{16} - 1092455T^{15} \\ 
\scriptstyle - 1114064T^{14} - 546399T^{13} - 148323T^{12} -
    45689T^{11} - 11976T^{10} - 6527T^9 - 3027T^8 - 1593T^7 \\ 
\scriptstyle - 464T^6 -
    65T^5 + 49T^4 + 30T^3 + 14T^2 + 4T + 1. 
\end{multline*}
On a larger scale,
for a random tetragonal genus $8$ curve over $\FF_{13^{6}}$ having Maroni invariants $\{1,2,2\}$ and Schreyer invariants $\{1,2\}$ we obtained its Hasse--Weil zeta function using about $1$ hour of computation.  For a random tetragonal genus $7$ curve over $\FF_{3^{16}}$ with Maroni invariants $\{1, 1, 2\}$ and Schreyer invariants $\{0,2\}$ we computed its zeta function in roughly $9$ hours. 
In both cases the lifting step took less than five seconds, of which the lion's share was accounted for by the resultant computation.

\subsection*{Degree $d=5$.} The accompanying Magma file \texttt{precomputed\_{}5.m}, which can be reproduced by running \texttt{precomputation\_{}5.m}, contains hard-coded expressions for Cayley's sextic resolvent and for the altered fundamental resolvent map $g'$ from Section~\ref{sec:pentagonal}. It is invoked by the file \texttt{lifting\_{}lowgonal\_{}5.m}, in which the user chooses a finite field $\FF_q$ of characteristic $p > 3$, along with a suitable sequence of nine integers $e_1, e_2, e_3, e_4$, $b_1, b_2, b_3, b_4, b_5$. Running the code
\begin{itemize} 
\item first generates a random degree $5$ cover $\overline{C} \to \PP^1$ over $\FF_q$ with Maroni invariants $e_1, e_2, e_3, e_4$ and Schreyer invariants $b_1, b_2, b_3, b_4, b_5$, of which it chooses a somewhat scrambled monic defining polynomial; this serves as test input for our lifting procedure,
\item next applies the Bhargava correspondence to this input, thereby procuding a quadruple of skew-symmetric matrices in $\FF_q[x]^{5 \times 5}$, as discussed in Section~\ref{sec:pentagonal},
\item finally naively lifts these matrices to characteristic zero, after which it considers their linear combination with coefficients $1, y_1, y_2, y_3$; then it takes the five $4 \times 4$ sub-Pfaffians of this linear combination which, after eliminating the variables $y_2, y_3$ and making the result monic in $y = y_1$, gives rise to a lift of the form~\eqref{eq:liftedpol}; this polynomial is then printed to a file \texttt{inputcurve\_{}5.m}, which can be loaded as input to Tuitman's \texttt{pcc\_{}p.m} or \texttt{pcc\_{}q.m} implementation.
\end{itemize}
E.g., over $\FF_{17}$, a run of our code generated the random pentagonal curve
\begin{multline*}
\scriptstyle y^5 + (5x^{10} + x^7 + 7x^6 + 8x^5 + 6x^4 + 12x^3 + 10x^2 + 6x + 11)y^4 \\ 
\scriptstyle + (10x^{20} + 4x^{17} + 11x^{16} + 15x^{15} + 14x^{14} + 8x^{13}
    + 16x^{12} + 5x^{11} + 3x^{10} \\ 
    \scriptstyle + 3x^9 + 7x^8 + 11x^7 + 15x^5 + 16x^4 + 4x^3 + 12x^2 + 8x + 9)y^3 \\ 
    \scriptstyle + (10x^{30} + 6x^{27} + 8x^{26} 
    + 14x^{25} + 6x^{24} + 3x^{23} + 5x^{22} + 13x^{21} + 7x^{20} + 3x^{19} + 6x^{18} + 13x^{17} + 14x^{16} + 11x^{15} \\ 
    \scriptstyle + 6x^{14} + 16x^{13} + 4x^{12} 
    + 13x^{11} + 5x^{10} + 2x^8 + 16x^7 + 11x^6 + 15x^5 + 16x^4 + 6x^3 + 3x^2 + 3x + 16)y^2 \\ 
    \scriptstyle + (5x^{40} + 4x^{37} + 11x^{36} + 
    15x^{35} + 11x^{34} + 13x^{33} + 2x^{32} + x^{31} + 15x^{30} + 14x^{29} + 2x^{28} + 6x^{27} \\ 
    \scriptstyle + 4x^{26} + 4x^{24} + 9x^{23} + 5x^{22} + 13x^{21} + 
    2x^{20} + 9x^{19} + 8x^{18} + 15x^{17} + 11x^{16} + 14x^{15} + 4x^{14} \\ 
    \scriptstyle + 4x^{13} + 7x^{11} + x^{10} + 9x^9 + 8x^8 + 11x^7 + 12x^6 + 6x^4 +
    14x^3 + 2x + 6)y \\ 
    \scriptstyle + x^{50} + x^{47} + 7x^{46} + 8x^{45} + 13x^{44} + 6x^{43} + 3x^{42} + 4x^{41} + 14x^{39} + 3x^{38} + 4x^{37} + 7x^{36} \\ 
    \scriptstyle + 
    12x^{35} + 12x^{34} + 9x^{33} + 11x^{32} + x^{30} + 11x^{29} + 2x^{28} + 9x^{27} + 11x^{26} + 8x^{25} + 7x^{24} \\ 
    \scriptstyle + 16x^{23} + 11x^{22} + x^{21} + 
    8x^{20} + 10x^{18} + 15x^{17} + 14x^{16} + 16x^{15}  + 4x^{14}  + 6x^{13} \\ 
    \scriptstyle + 3x^{11} + 5x^{10} + 4x^9 + 9x^8 + 15x^7 + 7x^6 + 10x^5 + 5x^4
    + 5x^3 + 12x^2 + 15x + 7
\end{multline*}
of genus $9$, having Maroni invariants $\{1,1,1,2\}$ and Schreyer invariants $\{0,1,1,2,2\}$. The Bhargava correspondence then produced the quadruple of skew-symmetric matrices
\begin{multline*}
\left(  \begin{smallmatrix} 0 &  16x^2 + 15x + 9 &  2x + 10 &     14x + 3 &   13 \\
 x^2 + 2x + 8 &  0 &  3 &  6x + 15 & 0 \\
15x + 7 & 14 & 0 &  5 &  0 \\
  3x + 14 &  11x + 2 &  12 &  0 &  0 \\
 4 & 0 & 0 &  0 &  0 \\ \end{smallmatrix} \right),
\left( \begin{smallmatrix} 
0 & 16x + 14 & 6x + 10  & 3x + 6 &  10 \\
x + 3 & 0 & 7x + 16 & 10x & 9 \\
 11x + 7 & 10x + 1 &  0 &  1 & 0 \\
14x + 11 & 7x &  16 &  0 & 0 \\
 7 &    8 &  0 & 0 &  0 \\
\end{smallmatrix} \right),
\\
\left( \begin{smallmatrix}  0  & x^2 + 3 &  14x + 11 &  16x + 11 &     10 \\
16x^2 + 14 &  0 &  10x + 14  &  13x + 1 &    16 \\
  3x + 6 & 7x + 3 &  0  & 3 &  0 \\
  x + 6  &  4x + 16 &    14 &    0&  0 \\
  7 &  1  &  0  & 0 &  0 \\
\end{smallmatrix} \right),
\\
\left( \begin{smallmatrix}
  0  &  8x^3 + x^2 + 12x + 6    &    12x + 8 & 11x^2 + 16x + 15 &    11x + 16 \\
9x^3 + 16x^2 + 5x + 11 & 0  &   4x^2 + 1    &    7x^2 + 8x + 11 &    5x + 5 \\
  5x + 9    &   13x^2 + 16  & 0 & 11x + 6 &   1 \\
 6x^2 + x + 2  &    10x^2 + 9x + 6 &  6x + 11 &    0    &    10 \\
 6x + 1     &    12x + 12   & 16 & 7&        0 \\
\end{smallmatrix} \right).
\end{multline*}
These matrices were then lifted to characteristic zero, i.e., to matrices over $\ZZ[x]$ whose entries have coefficients in $\{-8, \ldots, 8\}$; note that this coefficient range forces the lifted matrices to be skew-symmetric. After taking the linear combination with coefficients $1, y = y_1, y_2, y_3$, computing the five $4 \times 4$ sub-Pfaffians of the resulting skew-symmetric matrix, eliminating the variables $y_2, y_3$ and making the outcome monic in $y$, we ended up with a polynomial of the form~\eqref{eq:liftedpol} which was fed as input to Tuitman's code.
The numerator of its Hasse--Weil zeta function was then determined to be
\begin{multline*}
\scriptstyle  118587876497T^{18}  
- 20927272323T^{17} + 4513725403T^{16} - 168962983T^{15} + 271192687T^{14} \\ \scriptstyle - 57044843T^{13} 
+ 12616584T^{12} - 3142008T^{11} + 924732T^{10} - 198240T^9 + 54396T^8 - 10872T^7 \\ \scriptstyle + 2568T^6 - 
    683T^5 + 191T^4 - 7T^3 + 11T^2 - 3T + 1.
\end{multline*}
This basic example took $7.5$ hours of computation; as mentioned in the introduction, this is due to coefficient growth during variable elimination. Nevertheless, it is feasible to reach non-trivial ranges. E.g., for a random pentagonal genus $7$ curve over $\FF_{211}$ having Maroni invariants $\{0,1,1,1\}$ and Schreyer invariants $\{0,0,0,1,1\}$ we obtained its Hasse--Weil zeta function using about 28 hours of computation. %Over $\FF_{11^3}$, the same choice of parameters required \textbf{to do} hours.
%As mentioned in the introduction, examples of cryptographic size (say $q^g \approx 2^{256}$) seem out of reach, for the time being.

\end{document}